\documentclass[11pt,reqno]{amsart}

\usepackage{amsmath,amsfonts,amssymb}
\usepackage{mathrsfs,graphicx,extpfeil}
\usepackage{indentfirst}
\usepackage{latexsym}
\usepackage{enumerate}
\usepackage{float}
\usepackage{epsfig}
\usepackage{subfigure}
\usepackage{colortbl}
\usepackage{bm}
\usepackage{ulem}
 
\title[Asymptotic limits of generalized Kuramoto models]{Uniform-in-time asymptotic limits of generalized Kuramoto models}

\author[H. Cho]{Hangjun Cho}
\address[H. Cho]{\newline Department of Mathematical Sciences  \newline Seoul National University, Seoul 08826, Republic of Korea}
\email{hj2math@snu.ac.kr}

\author[S.-Y. Ha]{Seung-Yeal Ha}
\address[Seung-Yeal Ha]{\newline Department of Mathematical Sciences and Research Institute of Mathematics \newline  Seoul National University, Seoul 08826, Republic of Korea}
\email{syha@snu.ac.kr}

\author[M. Kang]{Myeongju Kang}
\address[M. Kang]{\newline School of Mathematics \newline Korea Institute for Advanced Study, Seoul 02455, Republic of Korea}
\email{mathemjkang@kias.re.kr}

\author[C. Min]{Chan Ho Min}
\address[C. Min]{\newline Department of Financial engineering \newline Ajou University, Suwon 16499, Republic of Korea}
\email{chanhomin@ajou.ac.kr}

\newtheorem{theorem}{Theorem}[section]
\newtheorem{lemma}{Lemma}[section]
\newtheorem{corollary}{Corollary}[section]
\newtheorem{proposition}{Proposition}[section]

\newtheorem{remark}{Remark}[section]

\newtheorem{definition}{Definition}[section]

\newcommand{\bbr}{\mathbb R}

\newcommand{\bbt} {\mathbb T}

\newcommand{\cD}{\mathcal D}
\newcommand{\cI}{\mathcal I}

\newcommand{\cT}{\mathcal T}
\newcommand{\cV}{\mathcal V}

\newcommand{\sgn}{\text{sgn}}

\begin{document}
%%%%%%%%%%%%%%%%

\date{\today}



\thanks{\textbf{Acknowledgment.} 
The work of S.-Y. Ha was supported by NRF-2020R1A2C3A01003881, the work of M. Kang was supported by a KIAS Individual Grant (MG091601) at Korean Institute for Advanced Study and the work of C. Min was supported by NRF-2021R1G1A1095140.
}

\begin{abstract}
We study two uniform-in-time asymptotic limits for generalized Kuramoto (GK) models. For these GK type models, we first derive the uniform stability estimates with respect to initial data, natural frequency and communication network under a suitable framework, and then as direct applications of this uniform stability estimate, we establish two asymptotic limits which are valid in the whole time interval, namely uniform-in-time continuum and mean-filed limit to the continuum and kinetic GK models, respectively. In the mean-field limit setting (the number of particles tends to infinity), we show global-in-time existence of measure-valued solutions to the corresponding kinetic equation. On the other hand, in a continuum limit setting (the lattice size tends to zero), we show that the lattice GKM solutions converge to a classical solution to the continuum GK model in supremum norm. Two asymptotic limits improve earlier results for the generalized GK type models.
%We study two uniform-in-time asymptotic limits for generalized Kuramoto (GK) models proposed in \cite{MAHK}.  For these GK type models, we first derive the uniform stability estimates with respect to initial data, natural frequency and communication network under a suitable framework formulated in terms of initial data and system parameters, and then as direct applications of this uniform stability estimate,  we establish two asymptotic limits which are valid in the whole time interval, namely uniform-in-time continuum and mean-filed limit to the continuum and kinetic GK models, respectively. In the mean-field limit setting (the number of particles tends to infinity),  we show  global-in-time existence of measure-valued solutions to the corresponding kinetic equation via empirical measures and uniform-in-time stability estimates. On the other hand, in a continuum limit setting (the lattice size tends to zero), we show that the lattice GKM solutions converge to a classical solution to  the continuum GK model in supremum norm in the whole time interval for some class of initial data. Two asymptotic limits improve earlier results for the generalized GK type models which are valid only in a finite-time interval. 
\end{abstract}
\maketitle \centerline{\date}

\section{Introduction} \label{sec:1}
\setcounter{equation}{0}
The jargon {\it ``Synchronization"} denotes an adjustment of rhythms in an ensemble of coupled oscillators due to weak interactions between them, and it often appears in oscillatory complex systems, e.g., synchronization of pacemaker cells and fireflies \cite{B-B,Ku2,Pe,P-R}, etc. In spite of ubiquitous presence in nature, it mathematical treatment was done only in a half century ago by two pioneers,  Arthur Winfree \cite{Wi2} and Yoshiki Kuramoto \cite{Ku2} (see \cite{C-Sm, D-M2, D-M3, D-M1, G-K-X-23, T-T, T-B} for other collective dynamics). Right after their pioneering works, phase models for synchronization were extensively studied in diverse disciplines such as  applied mathematics, biology, control theory and physics \cite{A-B, B-C-M1, B-C-M, B-D-P, C-H-J-K, C-S, D-X, D-B1, D-B, Er, H-P-R-S}.  To set up the stage, we begin with a brief description of the GK type models and their formal asymptotic limits, continuum and kinetic equations. 

Let $F:(-L, L) \longrightarrow \bbr$ be a ${\mathcal C}^2$-function satisfying the following structural conditions:
\begin{equation} \label{A-0-0}
F(-\theta) = -F(\theta) \quad \mbox{and} \quad F^{\prime}(\theta) > 0, \quad \forall~ \theta \in (-L, L),
\end{equation}
and let  $[N] := \{1, \cdots, N \}$ be the set of particles, here we identify the number $\alpha$ with $\alpha$-th particle or  oscillator at $\alpha$ lattice point.

Consider a phase ensemble consisting of $N$ GK oscillators. Let $\theta_\alpha = \theta_\alpha(t)$ be the phase of the $\alpha$-th GK oscillator on unit sphere. Then, we assume that the temporal dynamics of $\theta_\alpha$ is governed by the following Cauchy problem to the GK model:
\begin{align}
\begin{cases} \label{sys:GKM}
\displaystyle F\big( \dot\theta_\alpha \big)=\nu_\alpha +\frac{\kappa}{N} \sum_{\beta=1}^N\phi_{\alpha\beta} \sin(\theta_\beta - \theta_\alpha), \quad \forall ~t>0, \vspace{.1cm} \\
\theta_\alpha(0)=\theta_\alpha^0, \quad \forall ~\alpha \in [N],
\end{cases}
\end{align}
where $\kappa>0$, $\phi_{\alpha \beta}>0$ and $\nu_\alpha$ are the uniform coupling strength, the communication weights and the natural frequency of $\alpha$-th oscillator satisfying 
\begin{equation} \label{A-0}
\phi_{\alpha \beta} = \phi_{\beta \alpha}>0,  \quad  F^{-1} \bigg( \nu_\alpha \pm\frac{\kappa}{N} \sum_{\beta=1}^N |\phi_{\alpha\beta}| \bigg) \in (-L, L), \quad \forall~\alpha, \beta \in [N].
\end{equation}
For the case $L = \infty$ and $F(\omega) = \omega$, the GK dynamics \eqref{sys:GKM}$_1$ reduces to the Kuramoto model. Another choice makes the relativistic Kuramoto model and we refer Section \ref{sec:2.2}. Emergent behaviors of \eqref{sys:GKM} have been studied in \cite{MAHK}. We also refer to \cite{A-B, B-C-M,  C-C-H-K-K, C-H-J-K, C-S, D-X, D-B1, D-B, H-J, H-K-R, H-P-R-S, P-R, St, V-W} for the emergent dynamics of the Kuramoto model.

We can specify the spatial location by introducing a region and a lattice. Let $\mu$ be a $d$-dimensional Lebesgue measure on $\bbr^d$ and $D\subset\bbr^d$ be a compact region with a positive measure. For given region $D$, let $\Gamma \subset D$ be a collection of all lattice points in $D$ with finite index set $\Lambda$. Then, one can rewrite $\Gamma$ as $\{x_\alpha\}_{\alpha\in\Lambda}$. Let $\theta_\alpha = \theta_\alpha(t)$ be the phase of (internal) GK oscillator located at the lattice point $x_\alpha$ and at time $t$. Then, the temporal evolution of phase $\theta_\alpha$ is governed by the following ``lattice GK model":
\begin{align} \label{sys:LGKM}
\begin{dcases}
F(\dot{\theta}_\alpha) = \nu_\alpha + \displaystyle\frac{\kappa}{|\Lambda|}\sum_{\beta\in\Lambda} \phi_{\alpha\beta}\sin(\theta_\beta - \theta_\alpha),\quad \forall ~t>0, \\
\theta_\alpha(0) = \theta^{0}_\alpha, \quad \forall~\alpha\in\Lambda,
\end{dcases}
\end{align}
where we assume that the system parameters satisfy the same structural assumptions as in \eqref{A-0}.  

Note that two models \eqref{sys:GKM} and \eqref{sys:LGKM} take the same structure except the index sets $[N]$ and $\Lambda$. This is why we can treat both systems \eqref{sys:GKM} and  \eqref{sys:LGKM}  in a common platform.
 \newline

In this paper, we address the following issues for the GK model \eqref{sys:GKM} and the lattice GK model \eqref{sys:LGKM}:
\begin{itemize}
\item (Q1): For $p\in[1, \infty]$, under what conditions on system parameters and initial data, are agent-based models \eqref{sys:GKM} and \eqref{sys:LGKM} uniformly $\ell_p$-stable with respect to initial data and system parameters?
\vspace{0.2cm}

\item (Q2) Can we rigorously derive the continuum GK model \eqref{sys:CGKM} from the lattice GK model \eqref{sys:LGKM} ?
\vspace{0.2cm}

\item (Q3) Can we rigorously derive the kinetic GK model \eqref{sys:KGKM} from the GK model \eqref{sys:GKM} ?
\end{itemize}
In Section \ref{sec:4} and Section \ref{sec:5}, we will see that the last two issues crucially rely on the first issue (Q1). Thus, the issue is whether uniform-in-time stability estimates hold for the aforementioned models or not. For the Kuramoto model $\eqref{sys:GKM}$ with $F(\omega) = \omega$, such a uniform stability estimates was successfully obtained in \cite{H-K-P-Z}. However, for our two models under consideration, due to the nonlinear nature of $F$, the same arguments employed in the Kuramoto model can not be employed in the current setting. Thus, the resolution of the uniform stability is the main novelty of this work.  \newline

The main results of this paper are three-fold. First, we are concerned with the uniform $\ell_p$-stability of the GK model \eqref{sys:GKM}. For this, we restrict the diameter of initial data on quadrant and show that when coupling strength $\kappa$ is sufficiently large, velocity of every oscillator converges to the same value exponentially fast (see Lemma \ref{L3.1}), and then, we use this exponential convergence to derive the uniform $\ell_p$-stability. More precisely, for two solutions $\Theta = \Theta(t)$ and $\tilde\Theta = \tilde\Theta(t)$ to \eqref{sys:GKM} subject to the initial data, natural frequency pair, and communication matrix pair $(\Theta^0, \cV, \Phi)$ and $(\tilde\Theta^0, \tilde\cV, \tilde\Phi)$, respectively and $p\in[1,\infty]$, we have
\begin{equation} \label{A-1}
\sup_{0 \leq t < \infty} \| \Theta(t) - \tilde\Theta(t) \|_p \lesssim \| \Theta^0 - \tilde\Theta^0 \|_p + \|\mathcal V - \tilde{\mathcal V}\|_p +\frac{\|\Phi - \tilde\Phi\|_p}{N^{\frac1p}},
\end{equation}
where we used handy notation $\Theta, {\mathcal V}$ and $\Phi$ (see Section \ref{sec:2.1}), and we refer to Corollary \ref{C3.1} for details. Uniform $\ell_p$-stability problem of the Kuramoto model has been extensively studied in literature \cite{C-H-J-K, H-K-P-R-S1,H-K-P-Z}. In these literature, Galilean invariance and constant velocity of the averaged process $\big( \sum_{\alpha=1}^N \dot\theta_\alpha = \sum_{\alpha=1}^N \nu_\alpha \big)$ are crucially used, and sufficient framework for uniform stability is formulated in terms of initial data and natural frequencies. Unfortunately, Galilean invariance and constant velocity of the averaged process do not hold for the GK model \eqref{sys:GKM} as it is. So we need additional conditions for the coupling strength $\kappa$ and more refined arguments. Moreover, we consider the effect of communication matrix on the stability to utilize it in continuum limit (see Section \ref{sec:4}). 

Second, we establish the continuum limit of  \eqref{sys:LGKM} by showing that a solution for the lattice GK model \eqref{sys:LGKM} converges to a solution of the corresponding continuum model \eqref{sys:CGKM} in supremum norm sense in the whole time interval (see Theorem \ref{T4.1}). More precisely, by letting the lattice size tend to zero, we can formally obtain a differential-integro equation for the phase field $\theta = \theta(t,x)$ on $\bbr_+ \times D$:
\begin{align} \label{sys:CGKM}
\begin{aligned}
\begin{cases}
F(\partial_t \theta(t,x)) = \nu(x) + \displaystyle\frac{\kappa}{\mu(D)}\int_D \phi(x,z)\sin(\theta(t,z)-\theta(t,x)) d\mu(z), \quad \forall ~t>0, \vspace{.1cm} \\
\theta(0,x) = \theta^0(x), \quad \mbox{a.e.} ~x \in D,
\end{cases}
\end{aligned}
\end{align}
which will be called the ``continuum GK model" hereafter. Here, $\nu = \nu(x)$ is the natural frequency function and $\phi = \phi(x,z)$ is the communication weight function, respectively. The function $\phi(x,z)$ is positive: $\phi(x,z)>0$ for all $x,z$ in $D$. For the special setting:
\begin{align*}
D = [0, 1], \quad L = \infty, \quad F(\omega) = \omega, \quad \phi \equiv 1,
\end{align*}
the continuum GK equation \eqref{sys:CGKM}$_1$ reduces to the continuum Kuramoto model studied in \cite{Er, Med, Tr1, Tr2}:
\begin{align} \label{sys:CKM}
\partial_t \theta(t,x) =\nu(x) + \kappa_1 \int_0^{1} \sin(\theta(t,y)-\theta(t,x)) dy, \quad \forall ~(t,x) \in \bbr_+ \times [0,1],
\end{align}
and the limiting process from the lattice Kuramoto model to the continuum Kuramoto model \eqref{sys:CKM} is also rigorously justified in \cite{Er, Med, Tr1, Tr2}. To the best of our knowledge, the rigorous derivation of continuum model from lattice model in $L^p$ sense for finite $p$ in a finite or whole time intervals has been treated in various works \cite{H-K-M1, H-K-M2, H-K-P-R-S2, Med} by using the Lebesgue differentiation theorem. However, under the a priori condition on domain $D$ (see Section \ref{sec:4.1}) and using the aforementioned uniform $\ell_p$-stability of the GK model, we can prove the uniform-in-time convergence in $L^{\infty}$. And then, we can also derive $L^\infty$-contraction of the solution operator. More precisely, let $\theta$ and $\tilde\theta$ be two solutions to \eqref{sys:CGKM} corresponding to initial data $\theta^0$ and $\tilde{\theta}^0$, respectively. Then, we have $L^\infty$-contraction (see Proposition \ref{P4.2}):
\begin{align*}
\sup_{t\geq0} \| \theta(t,\cdot) - \tilde\theta(t,\cdot) \|_{L^\infty(D)}  \le\| \theta^0 - \tilde\theta^0 \|_{L^\infty(D)}.
\end{align*}

Third, we derive the kinetic GK model \eqref{sys:KGKM} from the GK model \eqref{sys:GKM} in the whole time interval using the methodology employed in \cite{H-K-P-R-S1,H-K-P-Z, H-K-Z} (see Theorem \ref{T5.1}). In this procedure, we also use the uniform $\ell_p$-stability \eqref{A-1} of the GK model. Consider the following situation:
 \[ \phi_{\alpha\beta} \equiv 1, \quad \alpha, \beta \in [N] \quad \mbox{and} \quad \lim_{x\to \pm L} F(x) = \pm\infty. \]
Letting $N$ tend to infinity, we use a formal BBGKY hierarchy argument with the GK model \eqref{sys:GKM}$_1$ to find a formal mean-field equation:
\begin{align} \label{sys:KGKM}
\begin{cases}
\displaystyle \partial_t \rho +\partial_\theta({\mathfrak L}[\rho]\rho) = 0, \quad \forall ~(t, \theta, \nu) \in \bbr_+\times\bbt\times\bbr, \vspace{.1cm}\\
\displaystyle {\mathfrak L}[\rho](t, \theta, \nu) := G\bigg( \nu +\kappa\int_{\bbt\times\bbr} \sin(\theta_*-\theta) \rho(t, \theta_*, \nu_*) d\theta_*d\nu_* \bigg),
\end{cases}
\end{align}
which will be called the ``kinetic GK model" henceforward. Here, $G$ is the inverse function of $F$. Then our third result deals with the verification of this limiting process rigorously (see Section \ref{sec:5}). For the Kuramoto model with $L=\infty$ and $F(\omega) = \omega$, the kinetic GK equation \eqref{sys:KGKM} reduces to the kinetic Kuramoto model in \cite{La}:
\begin{align} \label{sys:KKM}
\begin{cases}
\displaystyle \partial_t \rho +\partial_\theta({\mathfrak L}[\rho]\rho) = 0, \quad \forall ~(t, \theta, \nu) \in \bbr_+\times\bbt\times\bbr, \vspace{.1cm}\\
\displaystyle {\mathfrak L}[\rho](t, \theta, \nu) := \nu +\kappa\int_{\bbt\times\bbr} \sin(\theta_*-\theta) \rho(t, \theta_*, \nu_*) d\theta_*d\nu_*,
\end{cases}\end{align}
and the limiting process from the Kuramoto model to the kinetic Kuramoto model \eqref{sys:KKM} is also rigorously justified in \cite{La} in any finite-time interval.

\vspace{.2cm}

The rest of this paper is organized as follows. In Section \ref{sec:2}, we review the previous results for the GK model, and  present basic estimates for the GK model. In Section \ref{sec:3}, we study the uniform $\ell_p$-stability estimate of the GK model which plays a key role in the following sections. In Section \ref{sec:4}, we present a sufficient framework for the uniform-in-time continuum limit and derive $L^\infty$-contraction of the resulting continuum GK model. In Section \ref{sec:5}, we also provide a sufficient framework for the uniform-in-time mean-field limit. Finally, Section \ref{sec:6} is devoted to a brief summary of our main results and  remaining issues for a future work. In Appendix \ref{App-A}, we provide a lengthy estimate arising from Lemma \ref{L3.2}. In Appendix \ref{App-B}, we provide a proof of Theorem \ref{T4.1}. In Appendix \ref{App-C}, we briefly present a finite-in-time continuum limit for the lattice GK model for a generic initial data. 

%%%%%%%%%%%%%%%%%%%%%%%%%%%%%
%
%
% Section 2
%
%
%%%%%%%%%%%%%%%%%%%%%%%%%%%%%
\section{Preliminaries} \label{sec:2}
\setcounter{equation}{0}
In this section, we first present several simplified notation to be used throughout the paper, and we recall previous results on the emergent dynamics of \eqref{sys:GKM}.  Finally, we provide several estimates to be used in later sections. 
\subsection{Gallery of notation} \label{sec:2.1} Below, we collect all simplified notation to be used throughout the paper. We set the state set and state matrix as follows.
\begin{align*}
& \Theta: = \{ \theta_\alpha \}_{\alpha=1}^{N} , \quad \Theta^0 := \{ \theta_\alpha^0 \}_{\alpha=1}^{N}, \quad \Omega := \{ \omega_\alpha \}_{\alpha=1}^{N}, \\
& \Omega^0 := \{ \omega^0_\alpha \}_{\alpha=1}^{N}, \quad  \mathcal{V}:=\{ \nu_\alpha \}_{\alpha=1}^{N}, \quad \Phi := (\phi_{\alpha\beta})_{1 \leq \alpha, \beta \leq N},
\end{align*}
where $\omega_i := {\dot \theta}_i$. Moreover, as long as there is no confusion, we also use the same notation $\Theta, \Omega = \dot\Theta, {\mathcal V}$ to denote state vectors as well:
\[ \Theta = (\theta_1, \cdots, \theta_N), \quad \Omega  = ( \omega_1, \cdots, \omega_N), \quad  \mathcal{V} = (\nu_1, \cdots, \nu_N). \]
For given $G = F^{-1}$, we choose constants $a_G, b_G \in (-L, L)$ such that
\begin{align*}
G \bigg( \nu_\alpha \pm\frac{\kappa}{N} \sum_{\beta=1}^N |\phi_{\alpha\beta}| \bigg) \in [a_G, b_G], \quad \forall ~\alpha \in [N],
\end{align*}
and we also define several constants as follows.
\begin{align*}
& m_{F'} := \min_{a_G \le x \le b_G} F'(x), \quad M_{F'} := \max_{a_G \le x \le b_G} F'(x), \quad m_{G'}:= \min_{F(a_G)\le y\le F(b_G)} G'(y), \\
& M_{G'} := \max_{F(a_G)\le y\le F(b_G)} G'(y), \quad  m_\Phi := \min_{1\leq \alpha, \beta \leq N} \phi_{\alpha\beta}, \quad  \nu_c := \frac1N \sum_{\alpha=1}^N \nu_\alpha.
\end{align*}
Next, we introduce $\ell_p$-norms and diameter functionals for finite-dimensional vector and matrix as follows.  Let  
\[ \mathbf x = (x_1, \cdots, x_d)\in\bbr^d, \quad A = (a_{\alpha\beta})_{1\leq \alpha, \beta \leq d}\in\bbr^{d\times d}, \quad p \in [1, \infty]. \]
Then, we define
\begin{align*}
\|\mathbf x\|_p :=
\begin{cases}
\displaystyle \bigg( \sum_{\alpha=1}^d |x_\alpha|^p \bigg)^{\frac1p}, & p\in[1, \infty), \vspace{.1cm}\\
\displaystyle \max_{1\leq\alpha\leq d} |x_\alpha|, & p=\infty,
\end{cases}
\quad \|A\|_p :=
\begin{cases}
\displaystyle \bigg( \sum_{\alpha=1}^d |a_{\alpha\beta}|^p \bigg)^{\frac1p}, & p\in[1, \infty), \vspace{.1cm}\\
\displaystyle \max_{1\leq\alpha\leq d} |a_{\alpha\beta}|, & p=\infty.
\end{cases}
\end{align*}
In addition, diameters $\cD(\mathbf x)$ and $\cD(A)$ are defined by
\begin{align*}
\cD(\mathbf x) := \max_{1\leq \alpha\leq d} x_\alpha -\min_{1\leq \alpha\leq d} x_\alpha, \quad \cD(A) := \max_{1\leq\alpha,\beta\leq d} a_{\alpha\beta} -\min_{1\leq\alpha,\beta\leq d} a_{\alpha\beta}.
\end{align*}
Similarly, for real-valued functions $f$ and $g$ defined on the sets $D$ and $D \times D$, respectively, we define
\begin{align*}
& \|f\|_p :=
\begin{cases}
\displaystyle \bigg( \int_D |f(x)|^p d\mu(x) \bigg)^{\frac1p}, & p\in[1, \infty), \vspace{.1cm}\\
\displaystyle \sup_{x\in D} |f(x)|, & p=\infty,
\end{cases} \\
& \|g\|_\infty := \sup_{x, y\in D} |g(x, y)|, \quad \|g\|_{L^\infty_xL^1_y} := \sup_{x\in D} \int_D |g(x, y)| d\mu(y),
\end{align*}
and diameter functionals $\cD(f)$ and $\cD(g)$ for the ranges of $f$ and $g$:
\begin{align*}
\cD(f) := \sup_{x\in D} f(x) -\inf_{x\in D} f(x), \quad \cD(g) := \sup_{x, y\in D} g(x, y) -\inf_{x, y\in D} g(x, y).
\end{align*}

\subsection{Previous results} \label{sec:2.2} 
In this subsection, we briefly review the previous results in \cite{AHKS,MAHK} on the emergent dynamics of  \eqref{sys:GKM} with all-to-all couplings $\phi_{\alpha \beta} = 1$.  Before we move on,  we briefly comment on the modeling aspect of the generalized model. Following the relativistic Kuramoto model in \cite{MAHK}, the authors derived  \eqref{sys:GKM} with the following explicit $F$:
\[   \Gamma(\omega):= \frac{1}{\sqrt{1 - \frac{|\omega|^2}{c^2}}}, \quad  F(\omega) = \omega \Gamma(\omega) \bigg(1+\frac{\Gamma(\omega)}{c^2}\bigg), \quad L = c. \]
One can check that this explicit form satisfies the relations \eqref{A-0-0}.  We also refer to \cite{B-H-K} for the one-dimensional consensus model with the same structure as in  \eqref{sys:GKM} which can be derived from the relativistic Cucker-Smale model. Moreover, system \eqref{sys:GKM} with a memory type functional form $F$ also appears in the fractional Kuramoto model (see \cite{H-J}). \newline

Now, we consider a homogeneous GK ensemble with the same natural frequency:
\begin{align*}
\nu_\alpha = \nu, \quad \forall ~\alpha \in [N].
\end{align*}
Then, the dynamics of the homogeneous ensemble $\Theta$ is governed by the following Cauchy problem:
\begin{equation}
\begin{cases} \label{B-0}
\displaystyle \dot\theta_\alpha  = G \Big(  \nu +\frac{\kappa}{N} \sum_{\beta=1}^N \sin(\theta_\beta -\theta_\alpha) \Big), \quad \forall~t>0, \\
\displaystyle \theta_\alpha(0) = \theta_\alpha^{0}, \quad \forall~\alpha \in [N].
\end{cases}
\end{equation}
where $G = F^{-1}$ is the inverse of $F$ satisfying \eqref{A-0-0}.  Then, the emergent dynamics of \eqref{B-0} can be summarized in the following proposition. 
\begin{proposition} \label{P2.1}
\cite{MAHK}
Suppose initial data and coupling strength satisfy
\[   {\mathcal D}(\Theta^{0}) < \pi \quad \mbox{and} \quad \kappa > 0, \]
and let $\Theta = \Theta(t)$ be a global smooth solution to \eqref{B-0}. Then, the following assertions hold.
\begin{enumerate}
\item
(Existence of a trapping set): The diameter functional ${\mathcal D}(\Theta)$ is contractive:
\[
{\mathcal D}(\Theta(t)) \leq{\mathcal D}(\Theta^{0}), \quad \forall ~t>0.
\]
\item
(Exponential synchronization):~ there exists a positive constant $\Lambda_0 = \Lambda_0(\Theta^{0}, \nu, \kappa, G^{\prime})$ such that  
\[ {\mathcal D}(\Theta(t)) \leq e^{-\Lambda_0 t} {\mathcal D}(\Theta^{in}), \quad \forall~t > 0. \]
\end{enumerate}
\end{proposition}
\begin{remark} \label{R2.1}
For the globally coupled Kuramoto model, the assertions in Proposition \ref{P2.1} were already studied in \cite{C-H-J-K}.  Moreover, the restriction on the initial data can also be removed to generic initial data (see \cite{D-X}). \end{remark}
Next, we consider a heterogeneous GK ensemble with distributed natural frequencies whose dynamics is governed by the following Cauchy problem:
\begin{align} \label{B-1}
\begin{cases}
\displaystyle \dot\theta_\alpha = G\bigg( \nu_\alpha +\frac{\kappa}{N} \sum_{\beta = 1}^N \sin(\theta_\beta -\theta_\alpha) \bigg), \quad \forall~t>0, \\
\displaystyle \theta_\alpha(0) = \theta_\alpha^{0}, \quad \forall~\alpha \in [N]. 
\end{cases}
\end{align}
Then the emergent dynamics of \eqref{B-1} can be summarized in the following proposition. 
\begin{proposition} \label{P2.2}
\cite{MAHK}
Suppose that initial data, natural frequency, and coupling strength satisfy
\[ \kappa > \mathcal D(\mathcal V) > 0, \quad 0<\mathcal D(\Theta^0) < \pi-\theta_*, \quad  \theta_* :=\sin^{-1} \bigg( \frac{{\mathcal D}(\mathcal V)}{\kappa} \bigg) \in\bigg( 0, ~\frac{\pi}{2} \bigg),
\]
and let $\Theta$ be the global smooth solution to \eqref{B-1}. Then, the following assertions hold.
\begin{enumerate}
\item
(Existence of a trapping set):~there exists $t_* \geq 0$ such that, for all $t\geq t_*$,
\[
D(\Theta(t)) \leq \max\{ \theta_*, ~\min\{ D(\Theta^{0}), ~\pi-D(\Theta^{0}) \}\} \leq \frac{\pi}{2}.
\]
\item
(Complete synchronization):~relative frequencies tend to zero as $t \to \infty$:
\[
\lim_{t\to\infty} \big| \dot\theta_i(t) -\dot\theta_j(t) \big| = 0.
\]
\end{enumerate}
\end{proposition}
\begin{remark} \label{R2.2}
As mentioned in Remark \ref{R2.1}, for the Kuramoto model, the above assertions were already obtained in \cite{C-H-J-K} , and the complete synchronization estimate (assertion (ii)) has been shown for generic data in a large coupling regime using the uniform boundedness of fluctuations around the average phase and gradient-flow formulation of the Kuramoto model (see \cite{H-K-R}).
\end{remark}

\subsection{Basic estimates} \label{sec:2.3} In this subsection, we provide several estimates for \eqref{sys:GKM}. Recall that in \cite{MAHK}, uniform boundedness of solutions and complete synchronization are obtained for all-to-all coupling, i.e. $\phi_{\alpha\beta} \equiv 1$. In what follows, we present analogous results for general symmetric network topology $\Phi = (\phi_{\alpha \beta})$.
\begin{lemma} \label{L2.1}
Suppose that there exists $\theta_* \in(0, \pi/2)$ such that initial data and system parameters satisfy
\begin{equation} \label{B-1-0}
\kappa  > \frac{\mathcal D(\cV)}{m_\Phi \sin\theta_*},  \quad \mathcal{D}(\Theta^0)\le\theta_*,
\end{equation}
and let $\Theta = \Theta(t)$ be a global smooth solution to \eqref{sys:GKM}. Then, the phase diameter is uniformly bounded:
\begin{equation} \label{B-1-1}
\sup_{0 \leq t < \infty}\mathcal{D}(\Theta(t)) \le \theta_*.
\end{equation}
\end{lemma}
\begin{proof}
We choose time-dependent indices $M$ and $m$ satisfying
$$
\theta_M := \max_{1\le \alpha \le N}\theta_\alpha,\quad \theta_m :=\min_{1\le \alpha \le N}\theta_\alpha,
$$
and let $\nu_M$ and $\nu_m$ be time-dependent natural frequencies of $M$-th and $m$-th oscillators, respectively. For \eqref{B-1-1}, we use the contradiction argument. We set
$$
\mathcal{T}:= \{t>0\,:\,\mathcal{D}(\Theta(t)) \leq \theta_* \},\quad t_*:=\sup\mathcal{T}.
$$
By $\eqref{B-1-0}_2$, one can see that the set  $\cT$ is nonempty. Hence $t_* \in (0, \infty]$. Next, we claim:
\begin{equation} \label{B-1-2}
t_* = \infty.
\end{equation}
{\it Proof of \eqref{B-1-2}}: suppose the contrary holds. Then, it follows from the definition of $t_*$ that
\begin{align} \label{B-1-3}
\mathcal{D}(\Theta(t_*))=\theta_*,\quad \frac{d^+}{dt}\bigg|_{t=t_*}\mathcal{D}(\Theta)\geq 0,
\end{align}
where we used the Dini derivative $\frac{d^+}{dt}$. \newline

Note that 
 \begin{align}
\begin{aligned} \label{B-1-4}
& \nu_M+\frac{\kappa}{N} \sum_{\beta=1}^N \phi_{\beta M} \sin\big( \theta_\beta(t_*) - \theta_M(t_*) \big)-\bigg(\nu_m+\frac{\kappa}{N} \sum_{\beta=1}^N \phi_{\beta m}\sin\big( \theta_\beta(t_*) - \theta_m(t_*) \big)  \bigg)\\
&\hspace{.2cm} \le \cD(\cV) - \frac{\kappa m_\Phi}{N} \sum_{\beta=1}^N \big(\sin\big( \theta_M(t_*) - \theta_\beta(t_*) \big) + \sin\big( \theta_\beta(t_*) - \theta_m(t_*) \big)  \big)\\
&\hspace{.2cm} < \cD(\cV) -\kappa m_\Phi \sin\big(\theta_M(t_*) - \theta_m(t_*) \big) =  \cD(\cV) -\kappa m_\Phi \sin\theta_* < 0,
\end{aligned}
\end{align}
where we used $\eqref{B-1-0}_1$ and the relation:
\[ 0<\sin(a+b)<\sin a+\sin b \quad \mbox{for any $a, b \in (0, \pi/2)$}. \]
Then, we use the mean-value theorem and \eqref{B-1-4} to find 
\begin{align*}
\begin{aligned}
\frac{d^+}{dt}\bigg|_{t=t_*}\mathcal{D}(\Theta) &= 
G\bigg(\nu_M+\frac{\kappa}{N} \sum_{\beta=1}^N \phi_{\beta M} \sin\big( \theta_\beta - \theta_M \big)\bigg)-G\bigg(\nu_m+\frac{\kappa}{N} \sum_{\beta=1}^N \phi_{\beta m}\sin\big( \theta_\beta - \theta_m\big)  \bigg)
\\
&< m_{G'} \big( \cD(\cV) -\kappa m_\Phi \sin\theta_* \big) < 0.
\end{aligned}
\end{align*}
This contradicts \eqref{B-1-3}. Thus we have $t_* = \infty$ and we have the desired result.
\end{proof}

\vspace{0.5cm}

We differentiate \eqref{sys:GKM}$_1$ to obtain a second-order formation of the GK model:
\begin{align}
\begin{cases} \label{sys:GKM2}
\displaystyle \dot\theta_\alpha = \omega_\alpha, \quad  \forall~t>0, \vspace{.1cm}\\
\displaystyle \dot\omega_\alpha = \frac{\kappa}{N} \sum_{\beta=1}^N \phi_{\alpha\beta}(\omega_\beta - \omega_\alpha)\frac{\cos(\theta_\beta - \theta_\alpha)}{F'(\omega_\alpha)}, \vspace{.1cm}\\
\displaystyle \theta_\alpha(0)=\theta_\alpha^0,\quad \omega_\alpha(0) :=  G \Big(  \nu_\alpha + \frac{\kappa}{N} \sum_{\beta = 1}^{N} \sin(\theta^0_\beta - \theta^0_\alpha)  \Big), \quad \forall~\alpha \in [N].
\end{cases}
\end{align}
In the following lemma, we study the relation between the first-order model \eqref{sys:GKM} and the second-order model \eqref{sys:GKM2}. 
\begin{lemma} \label{L2.2}
Two Cauchy problems \eqref{sys:GKM} and \eqref{sys:GKM2} are equivalent. 
\end{lemma}
\begin{proof}
\noindent (i)~First, let $\Theta = \Theta(t)$ be a global smooth solution to \eqref{sys:GKM} with the initial data set $(\Theta^0, {\mathcal V})$. Then, it is clear to check that $(\Theta, \Omega = {\dot \Theta})$ is a solution to \eqref{sys:GKM2}. \newline

\noindent (ii)~Let $(\Theta, \Omega) = (\Theta(t), \Omega(t))$ be a global smooth solution to \eqref{sys:GKM2} with the initial data $(\Theta^0, \Omega^0)$. We can rewrite the \eqref{sys:GKM2}$_2$ by
\begin{align*}
F'(\omega_\alpha)\dot\omega_\alpha = \frac{\kappa}{N} \sum_{\beta=1}^N \phi_{\alpha\beta}(\omega_\beta - \omega_\alpha) \cos(\theta_\beta - \theta_\alpha).
\end{align*}
By integrating both sides with respect to $t$, one obtains
\begin{align*}
F(\omega_\alpha(t)) = \bigg[ F(\omega_\alpha^0) -\frac{\kappa}{N} \sum_{\beta=1}^N \phi_{\alpha\beta} \sin(\theta_\beta^0 - \theta_\alpha^0) \bigg] +\frac{\kappa}{N} \sum_{\beta=1}^N \phi_{\alpha\beta} \sin(\theta_\beta(t) - \theta_\alpha(t)).
\end{align*}
This implies that $\Theta$ is the solution of \eqref{sys:GKM} with the initial data and natural frequency pair $(\Theta^0, \mathcal V)$:
\[  \nu_\alpha := F(\omega_\alpha(0)) -  \frac{\kappa}{N} \sum_{\beta = 1}^{N} \sin(\theta^0_\beta - \theta^0_\alpha).  \]
\end{proof}
\begin{lemma} \label{L2.3}
Suppose that there exists $\theta_*\in(0, \pi/2)$ such that initial data and system parameters satisfy
\[
\kappa  > \frac{\mathcal D(\cV)}{ m_\Phi \sin\theta_*}, \quad \mathcal{D}(\Theta^0)\le\theta_*,
\]
and let $\Theta = \Theta(t)$ be a global smooth solution to \eqref{sys:GKM}. Then, there exists a positive constant $\Lambda_1 = \Lambda_1(\kappa, \Phi, F, \theta_*)$ such that
\begin{align*}
\mathcal D(\Omega(t)) \le \mathcal D(\Omega^0) e^{-\Lambda_1 t},\quad \forall ~t > 0.
\end{align*}
\end{lemma}

\begin{proof}
As in the proof of Lemma \ref{L2.1}, we choose time-dependent indices $M$ and $m$ satisfying
\[
\omega_M := \max_{1\le \alpha \le N}\omega_\alpha,\quad \omega_m :=\min_{1\le \alpha\le N}\omega_\alpha,
\]
Then for a differentiable point $t \in (0, \infty)$, we use Lemma \ref{L2.1} and \ref{L2.3} to observe the behavior of $\cD(\Omega)$:
\begin{align*}
\begin{aligned}
\frac{d}{dt} \big( \omega_M - \omega_m \big) &=
\frac{\kappa}{N} \sum_{\beta=1}^N \phi_{M\beta}(\omega_\beta - \omega_M)\frac{\cos(\theta_\beta - \theta_M)}{F'(\omega_M)} -
\frac{\kappa}{N} \sum_{\beta=1}^N \phi_{m\beta}(\omega_\beta - \omega_m)\frac{\cos(\theta_\beta - \theta_m)}{F'(\omega_m)}\\
&\le \frac{\kappa m_\Phi}{N} \sum_{\beta=1}^N(\omega_\beta - \omega_M)\frac{\cos\mathcal D(\Theta)}{M_{F'}} +
\frac{\kappa m_\Phi}{N} \sum_{\beta=1}^N (\omega_m - \omega_\beta)\frac{\cos\mathcal D(\Theta)}{M_{F'}}\\
&= -\frac{\kappa m_{\Phi} \cos\mathcal D(\Theta)}{M_{F'}} \big( \omega_M - \omega_m \big).
\end{aligned}
\end{align*}
This implies our desired result with
\begin{align*}
\Lambda_1 = \frac{\kappa m_{\Phi} \cos\theta_*}{M_{F'}} >0.
\end{align*}
\end{proof}

%%%%%%%%%%%%%%%%%%%%%%%%%%%%%
%
%
% Section 3
%
%
%%%%%%%%%%%%%%%%%%%%%%%%%%%%%
\section{Uniform $\ell_p$-stability estimate} \label{sec:3}
\setcounter{equation}{0}
In this section, we provide a sufficient framework for uniform $\ell_p$-stability of the GK model \eqref{sys:GKM}.  The same argument can be applied to the lattice model \eqref{sys:LGKM}.

\subsection{A sufficient framework and uniform stability} \label{sec:3.1} 
First, we begin with Newton's quotient-like functional ${\mathcal Q}[\cdot]$:
\begin{align} \label{C-1}
& {\mathcal Q}^{-1}[F](x,y) :=
\begin{dcases}
\displaystyle\frac{x-y}{F(x) -F(y)} & x\neq y, \\
\displaystyle\frac{1}{F'(x)} & x=y,
\end{dcases} \quad 
{\mathcal Q}[G^{\prime}](x,y) :=
\begin{dcases}
\displaystyle \dfrac{G'(x) - G'(y)}{x-y} & x\neq y, \\
\displaystyle G''(x) & x=y.
\end{dcases}
\end{align} 
Since $F'>0$, these are well defined and note that $  {\mathcal Q}^{-1}[F]$ and ${\mathcal Q}[G^{\prime}]$ are symmetric functions in their arguments. Moreover, for later usage, we define
\begin{align} \label{C-2}
\begin{aligned}
& m_{ {\mathcal Q}^{-1}[F]} :=\min_{(x,y)\in[a_G,b_G]^2}  {\mathcal Q}^{-1}[F](x,y),\quad M_{ {\mathcal Q}^{-1}[F]} :=\max_{(x,y)\in[a_G,b_G]^2}  {\mathcal Q}^{-1}[F](x,y),\\
& m_{{\mathcal Q}[G^{\prime}]} := \min_{(x,y)\in[F(a_G),F(b_G)]^2} \big| {\mathcal Q}[G^{\prime}](x,y) \big|,~~M_{{\mathcal Q}[G^{\prime}]} := \max_{(x,y)\in[F(a_G),F(b_G)]^2} \big| {\mathcal Q}[G^{\prime}](x,y) \big|.
\end{aligned}
\end{align}
Next, we state a sufficient framework $({\mathcal F}_A)$ leading to the uniform $\ell_p$-stability.  Let $(\Theta^0, \cV, \Phi)$ and $(\tilde\Theta^0, \tilde\cV, \tilde\Phi)$ be triples consisting of initial data, natural frequency, and network topologies, respectively. We assume that there exist $b_\Phi \geq a_\Phi>0$, $-L < a_G < b_G < L$, and $\theta_*\in(0, \pi/2)$ such that 
\begin{equation} \label{C-3}
({\mathcal F}_A):~
\begin{cases}
& \displaystyle \max\{ \mathcal{D}(\Theta^0), \mathcal{D}(\tilde\Theta^0) \} \le\theta_*, \quad  a_\Phi \leq \phi_{\alpha\beta}, \tilde\phi_{\alpha\beta} \leq b_\Phi,~~\forall~\alpha,\beta \in [N], \\
& \displaystyle  \kappa  > \max \Big \{ \frac{\mathcal D(\cV)}{ m_\Phi \sin\theta_*},~~ \frac{\mathcal D(\tilde\cV)}{  m_{\tilde\Phi} \sin\theta_*},~~ \frac{M_{{\mathcal Q}^{-1}[F]}}{  \max\{m_\Phi, m_{\tilde\Phi}\} m_{G'}^2\cos\theta_*} \Big \}, \\
&\displaystyle  \quad F(a_G) +\kappa b_\Phi \leq \nu_\alpha, \tilde\nu_\alpha \leq F(b_G) -\kappa b_\Phi, \quad \sum_{\alpha=1}^N \nu_\alpha = \sum_{\alpha=1}^N \tilde\nu_\alpha.
\end{cases}
\end{equation}
Now we are ready to state the first main result on the uniform $\ell_p$-stability. 
\begin{theorem} \label{T3.1}
Suppose that the framework $({\mathcal F}_A)$ holds, and let Let $\Theta = \Theta(t)$ and $\tilde\Theta = \tilde\Theta(t)$ be two global solutions of \eqref{sys:GKM} subject to the initial data, natural frequency, and communication matrix triples $(\Theta^0, \cV, \Phi)$ and $(\tilde\Theta^0, \tilde\cV, \tilde\Phi)$, respectively. Then there exists a positive constant $\Lambda_2 = \Lambda_2(\kappa, F, a_\Phi, b_\Phi, a_G, b_G, \theta_*)$ such that for $p\in[1,\infty]$,
\begin{align*}
& \| \Theta(t) - \tilde\Theta(t) \|_p + \| \Omega(t) - \tilde\Omega(t) \|_p \leq \Lambda_2 \bigg( \| \Theta^0 - \tilde\Theta^0 \|_p + \| \Omega^0 - \tilde\Omega^0 \|_p +\frac{\|\Phi - \tilde\Phi\|_p}{N^{\frac1p}} \bigg), \quad \forall ~t\geq0.
\end{align*}
\end{theorem}
\begin{proof} Since the proof is very lengthy, we leave its proof in next two subsections.
\end{proof}

\subsection{Preparatory lemmas} \label{sec:3.2}
In this subsection, we present two preparatory lemmas for the proof of Theorem \ref{T3.1}.  
\begin{lemma} \label{L3.1}
Suppose initial data and system parameters satisfy  the following relations: there exists $\theta_*\in(0, \pi/2)$ such that
\begin{align*}%\label{C-3-1}
\kappa  > \frac{\mathcal D(\cV)}{ m_\Phi \sin\theta_*}, \quad \mathcal{D}(\Theta^0)\le\theta_*,
\end{align*}
and let $\Theta = \Theta(t)$ be a global smooth solution to \eqref{sys:GKM}. Then, we have
\begin{align*}
\big| F(\omega_\alpha(t)) - \nu_c \big| \le \frac{\mathcal{D}(\Omega^0)}{m_{ {\mathcal Q}^{-1}[F]}} e^{-\Lambda_1 t}, \quad \big| \omega_\alpha(t)-G(\nu_c) \big| \le \frac{M_{ {\mathcal Q}^{-1}[F]}}{m_{ {\mathcal Q}^{-1}[F]}} \mathcal{D}(\Omega^0) e^{-\Lambda_1 t},
\end{align*}
for all $t>0,~\alpha \in [N]$ where $\Lambda_1$ is a positive constant obtained in Lemma \ref{L2.3}.
\end{lemma}
\begin{proof}
We choose time-dependent indices $M$ and $m$ such that
\begin{align*}
\omega_M :=\max_{1\le \alpha \le N}\omega_\alpha, \quad \omega_m :=\min_{1\le \alpha \le N}\omega_\alpha.
\end{align*}
On the other hand, we add \eqref{sys:GKM} with respect to $\alpha$ to obtain
\begin{align} \label{C-3-2}
\frac{1}{N}\sum_{\alpha=1}^N F(\omega_\alpha) = \frac{1}{N} \sum_{\alpha=1}^N \nu_\alpha =: \nu_c
\end{align}
Hence, Lemma \ref{L2.3} yields
\begin{align*}
\big| F(\omega_\alpha) -\nu_c \big| \le F(\omega_M) - F(\omega_m)
= \frac{\mathcal D(\Omega)}{{\mathcal Q}^{-1}[F](\omega_M,\omega_m)}\le\frac{\mathcal{D}(\Omega^0)}{m_{{\mathcal Q}^{-1}[F]}} e^{-\Lambda_1 t}.
\end{align*}
This implies
\begin{align*}
\big| \omega_\alpha-G(\nu_c) \big| = \big| {\mathcal Q}^{-1}[F] \big( \omega_\alpha, G(\bar\nu) \big) \big| \cdot \big| F(\omega_\alpha) -\nu_c \big| \le  \frac{M_{ {\mathcal Q}^{-1}[F]}}{m_{ {\mathcal Q}^{-1}[F]}}  \mathcal{D}(\Omega^0) e^{-\Lambda_1 t}.
\end{align*}
\end{proof}

\vspace{0.2cm}

Now, we return to the uniform $\ell_p$-estimate. If the desired statement holds for all finite $p$, one can show the case of $p=\infty$ by letting $p$ go to infinity. This limiting argument is possible since $\Lambda_2$ does not depend on $p$. Moreover, since the case  $p=1$ can be proved in similar way, in this proof, we only consider $p\in(1,\infty)$. Without loss of generality, we assume
\[  m_\Phi \geq m_{\tilde\Phi}. \] We begin with differentiating $\|\Omega-\tilde\Omega\|_p^p$:
\begin{align}
\begin{aligned} \label{C-4}
&\frac{d}{dt} \sum_{\alpha=1}^N  \big|\omega_\alpha - \tilde\omega_\alpha \big|^p = p\sum_{\alpha=1}^N \big( \dot \omega_\alpha - \dot{\tilde\omega}_\alpha \big) \sgn\big( \omega_\alpha - \tilde\omega_\alpha \big) \big| \omega_\alpha - \tilde\omega_\alpha \big|^{p-1}\\
&= \frac{p\kappa}{N}\sum_{\alpha=1}^N \sgn\big( \omega_\alpha - \tilde\omega_\alpha \big) \big| \omega_\alpha - \tilde\omega_\alpha \big|^{p-1} \\
&\hspace{.2cm}\times \bigg[ \sum_{\beta=1}^N \bigg( \frac{\phi_{\alpha\beta} \cos(\theta_\beta - \theta_\alpha)}{F'(\omega_\alpha)} (\omega_\beta - \omega_\alpha) -\frac{\tilde\phi_{\alpha\beta} \cos(\tilde\theta_\beta - \tilde\theta_\alpha)}{F'(\tilde\omega_\alpha)} (\tilde\omega_\beta - \tilde\omega_\alpha)\bigg) \bigg] \\
&= \frac{p\kappa}{N}\sum_{\alpha,\beta=1}^N \phi_{\alpha\beta} \sgn\big( \omega_\alpha - \tilde\omega_\alpha \big) \big| \omega_\alpha - \tilde\omega_\alpha \big|^{p-1}  \\
&\hspace{1.8cm}\times\bigg(  \frac{\cos(\theta_\beta - \theta_\alpha)}{F'(\omega_\alpha)} (\omega_\beta - \omega_\alpha) -  \frac{\cos(\tilde\theta_\beta - \tilde\theta_\alpha)}{F'(\omega_\alpha)} (\omega_\beta - \omega_\alpha) \bigg) \\
&\hspace{.2cm}+\frac{p\kappa}{N}\sum_{\alpha,\beta=1}^N \phi_{\alpha\beta} \sgn\big( \omega_\alpha - \tilde\omega_\alpha \big) \big| \omega_\alpha - \tilde\omega_\alpha \big|^{p-1}  \\
&\hspace{1.8cm}\times\bigg( \frac{\cos(\tilde\theta_\beta - \tilde\theta_\alpha)}{F'(\omega_\alpha)} (\omega_\beta - \omega_\alpha) - \frac{\cos(\tilde\theta_\beta - \tilde\theta_\alpha)}{F'(\tilde\omega_\alpha)} (\omega_\beta - \omega_\alpha) \bigg) \\
&\hspace{.2cm}+\frac{p\kappa}{N}\sum_{\alpha,\beta=1}^N \phi_{\alpha\beta} \sgn\big( \omega_\alpha - \tilde\omega_\alpha \big) \big| \omega_\alpha - \tilde\omega_\alpha \big|^{p-1} \\
&\hspace{1.8cm}\times\bigg( \frac{\cos(\tilde\theta_\beta - \tilde\theta_\alpha)}{F'(\tilde\omega_\alpha)} (\omega_\beta - \omega_\alpha) - \frac{\cos(\tilde\theta_\beta - \tilde\theta_\alpha)}{F'(\tilde\omega_\alpha)} (\tilde\omega_\beta - \tilde\omega_\alpha) \bigg)\\
&\hspace{.2cm}+\frac{p\kappa}{N}\sum_{\alpha,\beta=1}^N \sgn\big( \omega_\alpha - \tilde\omega_\alpha \big) \big| \omega_\alpha - \tilde\omega_\alpha \big|^{p-1} \big( \phi_{\alpha\beta} -\tilde\phi_{\alpha\beta} \big) \frac{\cos(\tilde\theta_\beta - \tilde\theta_\alpha)}{F'(\tilde\omega_\alpha)} (\tilde\omega_\beta - \tilde\omega_\alpha) \\
& =: \mathcal{I}_{11} + \mathcal{I}_{12} + \mathcal{I}_{13} +\cI_{14}.
\end{aligned}
\end{align}
In the following lemma, we provide estimates for $\cI_{1i}$'s. 
\begin{lemma} \label{L3.2}
Suppose the same assumption in Theorem \ref{T3.1}. Then, the quantities ${\mathcal I}_{1i},~i=1, \cdots 4$, defined in \eqref{C-4} satisfy the following estimates.
\begin{align*}
\begin{aligned}
& (i) ~|\mathcal{I}_{11}| \leq \frac{6p\kappa^2 b_\Phi^2 M_{G'}}{m_{F'}} e^{-\Lambda_1 t}  \|\Omega -\tilde\Omega\|_p^{p-1} \| \Theta - \tilde\Theta \|_p, \\
&\hspace{.6cm} |\mathcal{I}_{12}| \leq  \frac{3p\kappa^2 b_\Phi^2 M_{G'} M_{{\mathcal Q}[G^{\prime}]}} {m_{{\mathcal Q}^{-1}[F]}} e^{-\Lambda_1 t}  \big\| \Omega - \tilde\Omega \big\|_p^{p},
\end{aligned}
\end{align*}
\begin{align*}
\begin{aligned}
& (ii)~ \mathcal{I}_{13} \leq  -\frac{p\kappa G'(\nu_c) C_{131}}{M_{{\mathcal Q}^{-1}[F]}}  \|\Omega - \tilde\Omega\|^p_p \\
&\hspace{1.4cm} +\bigg( \frac{12p\kappa^2 b_\Phi^2 M_{G'}^2 M_{{\mathcal Q}[G^{\prime}] }}{m^2_{{\mathcal Q}^{-1}[F]}} +\frac{6p\kappa^2 b_\Phi^2 M_{G'} M_{{\mathcal Q}[G^{\prime}]}}{m_{{\mathcal Q}^{-1}[F]}} \bigg) e^{-\Lambda_1t} \|\Omega - \tilde\Omega\|^{p}_p. \\
& (iii)~|\cI_{14}| \leq  \frac{3p\kappa^2 b_\Phi M_{G'}}{N^{\frac1p}m_{F'}}  e^{-\Lambda_1t} \|\Omega - \tilde\Omega\|^{p-1}_p \|\Phi - \tilde\Phi\|_p.
\end{aligned}
\end{align*}
where $C_{131}=G'(\nu_c)m_{\Phi}\cos \theta_*$.
\end{lemma}
\begin{proof}
\noindent (i)~We use Lemma \ref{L2.3} to obtain
\begin{align*}
\begin{aligned} \label{C-5}
|\mathcal{I}_{11}| &= \frac{p\kappa}{N}\bigg|\sum_{\alpha,\beta=1}^N \phi_{\alpha\beta} \sgn\big( \omega_\alpha - \tilde\omega_\alpha \big) \big| \omega_\alpha - \tilde\omega_\alpha \big|^{p-1} \frac{\omega_\beta - \omega_\alpha}{F'(\omega_\alpha)} \big[ \cos(\theta_\beta - \theta_\alpha)  -  \cos(\tilde\theta_\beta - \tilde\theta_\alpha)  \big]\bigg|\\
&\le \frac{p\kappa b_\Phi}{N m_{F'}} \mathcal{D}(\Omega^0) e^{-\Lambda_1 t} \sum_{\alpha,\beta=1}^N  \big| \omega_\alpha - \tilde\omega_\alpha \big|^{p-1} \big|  \cos(\theta_\beta - \theta_\alpha)  -  \cos(\tilde\theta_\beta - \tilde\theta_\alpha)  \big|\\
&\le \frac{p\kappa b_\Phi}{N m_{F'}} \mathcal{D}(\Omega^0) e^{-\Lambda_1 t} \sum_{\alpha,\beta=1}^N  \big| \omega_\alpha - \tilde\omega_\alpha \big|^{p-1} \big( | \theta_\alpha - \tilde\theta_\alpha | + | \theta_\beta - \tilde\theta_\beta|\big)\\
&\le \frac{2p\kappa b_\Phi}{m_{F'}} \mathcal{D}(\Omega^0) e^{-\Lambda_1 t}  \|\Omega -\tilde\Omega\|_p^{p-1} \| \Theta - \tilde\Theta \|_p.
\end{aligned}
\end{align*}
For the last inequality, we used the H\"older inequality to find
\begin{align*} 
\begin{aligned} 
& \sum_{\alpha,\beta=1}^N |\omega_\alpha - \tilde\omega_\alpha |^{p-1}\big( |\theta_\beta - \tilde\theta_\beta| + |\theta_\alpha - \tilde\theta_\alpha| \big)\\
&\le \bigg( \sum_{\alpha,\beta=1}^N |\omega_\alpha - \tilde\omega_\alpha |^{p-1\frac{p}{p-1}}  \bigg)^{\frac{p-1}{p}} \bigg(\sum_{\alpha,\beta=1}^N |\theta_\beta - \tilde\theta_\beta|^p \bigg)^\frac{1}{p} \\
&\hspace{.2cm}+N \bigg( \sum_{\alpha=1}^N  \big| \omega_\alpha - \tilde\omega_\alpha \big|^{p-1\cdot \frac{p}{p-1}}\bigg)^{\frac{p-1}{p}}\bigg( \sum_{\alpha=1}^N \big| \theta_\alpha - \tilde\theta_\alpha \big|^p  \bigg)^{\frac{1}{p}} \\
&= 2N \|\Omega -\tilde\Omega\|_p^{p-1} \| \Theta - \tilde\Theta \|_p.
\end{aligned}
\end{align*}
On the other hand, it follows from \eqref{sys:GKM2} and \eqref{C-3} that
\begin{equation} \label{C-6}
\cD(\Omega^0) \leq M_{G'} (\cD(\cV) +2\kappa b_\Phi) \leq  M_{G'} (\kappa m_\Phi + 2\kappa b_\Phi) < 3\kappa b_\Phi M_{G'}.
\end{equation}
This \eqref{C-6} implies
\begin{align*}
|\cI_{11}| \leq \frac{6p\kappa^2 b_\Phi^2 M_{G'}}{m_{F'}} e^{-\Lambda_1 t}  \|\Omega -\tilde\Omega\|_p^{p-1} \| \Theta - \tilde\Theta \|_p.
\end{align*}
Again, we use Lemma \ref{L2.3} to see
\begin{align*}
\begin{aligned}
|\mathcal{I}_{12}| &\leq \frac{p\kappa}{N} \sum_{\alpha,\beta=1}^N \bigg| \phi_{\alpha\beta} \big| \omega_\alpha - \tilde\omega_\alpha \big|^{p-1} (\omega_\beta - \omega_\alpha) \cos(\tilde\theta_\beta - \tilde\theta_\alpha) \bigg( \frac{1}{F'(\omega_\alpha)}  -\frac{1}{F'(\tilde\omega_\alpha)} \bigg) \bigg| \\
&\leq \frac{p\kappa b_\Phi}{N} \mathcal{D}(\Omega^0) e^{-\Lambda_1 t}  \sum_{\alpha,\beta=1}^N \big| \omega_\alpha - \tilde\omega_\alpha \big|^{p-1} \bigg| \frac{1}{F'(\omega_\alpha)}  -\frac{1}{F'(\tilde\omega_\alpha)} \bigg| \\
&\le \frac{3p\kappa^2 b_\Phi^2 M_{G'} M_{{\mathcal Q}[G^{\prime}]}}{m_{ {\mathcal Q}^{-1}[F]}} e^{-\Lambda_1 t}  \big\| \Omega - \tilde\Omega \big\|_p^{p},
\end{aligned}
\end{align*}
where the last inequality can be obtained by
\begin{align*}
\begin{aligned}
\bigg| \frac{1}{F'(\omega_\alpha)}  -\frac{1}{F'(\tilde\omega_\alpha)} \bigg| &= \big| G'(F(\omega_\alpha)) -G'(F(\tilde\omega_\alpha)) \big|  \\
&= \big| {\mathcal Q}[G^{\prime}]\big( F(\omega_\alpha), F(\tilde\omega_\alpha) \big) \big| \frac{|\omega_\alpha -\tilde\omega_\alpha|}{| {\mathcal Q}^{-1}[F](\omega_\alpha, \tilde\omega_\alpha)|} \leq \frac{M_{{\mathcal Q}[G^{\prime}]}} {m_{{\mathcal Q}^{-1}[F]} }\big| \omega_\alpha -\tilde\omega_\alpha \big|.
\end{aligned}
\end{align*}

\vspace{0.2cm}

\noindent (ii)~Since the estimate for ${\mathcal I}_{13}$ is very lengthy, we leave it to Appendix \ref{App-A}. \newline

\noindent (iii)~ It follows from Lemma \ref{L2.3} that
\begin{align*}
|\cI_{14}| &\leq \frac{p\kappa}{N} \sum_{\alpha,\beta=1}^N \bigg| \big| \omega_\alpha - \tilde\omega_\alpha \big|^{p-1} \big( \phi_{\alpha\beta} -\tilde\phi_{\alpha\beta} \big) \frac{\cos(\tilde\theta_\beta - \tilde\theta_\alpha)}{F'(\tilde\omega_\alpha)} (\tilde\omega_\beta - \tilde\omega_\alpha)\bigg| \\
&\leq \frac{p\kappa}{Nm_{F'}} \mathcal{D}(\tilde\Omega^0) e^{-\Lambda_1t} \sum_{\alpha,\beta=1}^N \big| \omega_\alpha - \tilde\omega_\alpha \big|^{p-1} \big| \phi_{\alpha\beta} -\tilde\phi_{\alpha\beta} \big| \\
&\leq \frac{3p\kappa^2 b_\Phi M_{G'}}{N^{\frac1p}m_{F'}}  e^{-\Lambda_1t} \|\Omega - \tilde\Omega\|^{p-1}_p \|\Phi - \tilde\Phi\|_p,
\end{align*}
where the last inequality can be obtained by \eqref{C-6} and
\begin{align*}
 \sum_{\alpha,\beta=1}^N \big| \omega_\alpha - \tilde\omega_\alpha \big|^{p-1} \big| \phi_{\alpha\beta}-\tilde\phi_{\alpha\beta} \big| &\leq \bigg( \sum_{\alpha,\beta=1}^N |\omega_\alpha - \tilde\omega_\alpha |^{p-1\frac{p}{p-1}}  \bigg)^{\frac{p-1}{p}} \bigg(\sum_{\alpha,\beta=1}^N \big| \phi_{\alpha\beta}-\tilde\phi_{\alpha\beta} \big|^p \bigg)^\frac{1}{p}  \\
& = N^{\frac{p-1}{p}} \|\Omega - \tilde\Omega\|^{p-1}_p \|\Phi - \tilde\Phi\|_p.
\end{align*}

\subsection{Proof of Theorem \ref{T3.1}} \label{sec:3.3} Now, we are ready to provide a proof of our first main result.

In \eqref{C-4}, we combine all the estimates for ${\mathcal I}_{1i}$ in Lemma \ref{L3.2} to get 
\begin{align*}
\begin{aligned}
 \frac{d}{dt} \sum_{\alpha=1}^N \big|\omega_\alpha - \tilde\omega_\alpha \big|^p &= \frac{d}{dt} \|\Omega - \tilde\Omega\|^p_p = p \|\Omega - \tilde\Omega\|^{p-1}_p \frac{d}{dt} \|\Omega - \tilde\Omega\|_p \\
&\leq -\frac{p\kappa G'(\nu_c) C_{131}}{M_H} \|\Omega - \tilde\Omega\|^p_p +p\Lambda_{21} e^{-\Lambda_1 t}  \|\Omega -\tilde\Omega\|_p^{p-1} \| \Theta - \tilde\Theta \|_p \\
&+p\Lambda_{22} e^{-\Lambda_1 t}  \big\| \Omega - \tilde\Omega \big\|_p^{p} +\frac{p\Lambda_{23}}{N^{\frac1p}} e^{-\Lambda_1t} \|\Omega - \tilde\Omega\|^{p-1}_p \|\Phi - \tilde\Phi\|_p,
\end{aligned}
\end{align*}
for some constants $\Lambda_{2i}=\Lambda_{2i}(\kappa,F,a_\Phi,b_\Phi,a_G,b_G)$, $i=1,2,3$. This implies
\begin{align*}
\frac{d}{dt} \|\Omega - \tilde\Omega\|_p &\leq -\frac{\kappa G'(\nu_c) C_{131}}{M_H} \|\Omega - \tilde\Omega\|_p +\Lambda_{21} e^{-\Lambda_1 t} \| \Theta - \tilde\Theta \|_p  \\
&+\Lambda_{22} e^{-\Lambda_1 t} \|\Omega - \tilde\Omega\|_p +\frac{\Lambda_{23}}{N^{\frac1p}} e^{-\Lambda_1t} \|\Phi - \tilde\Phi\|_p.
\end{align*}
On the other hand, we use the H\"older inequality to get
\begin{align*}
\begin{aligned}
& p \| \Theta - \tilde\Theta \|_p^{p-1}\frac{d}{dt}\| \Theta - \tilde\Theta \|_p \\
& \hspace{0.5cm} = \frac{d}{dt} \bigg( \sum_{\alpha=1}^N \big| \theta_\alpha - \tilde\theta_\alpha \big|^p \bigg) = p\sum_{\alpha=1}^N \big( \omega_\alpha - \tilde\omega_\alpha \big) \sgn\big( \theta_\alpha - \tilde\theta_\alpha \big) \big| \theta_\alpha - \tilde\theta_\alpha \big|^{p-1}\\
&\hspace{0.5cm} \le p\sum_{\alpha=1}^N \big| \omega_\alpha - \tilde\omega_\alpha \big| \big| \theta_\alpha - \tilde\theta_\alpha \big|^{p-1} = p \| \Omega - \tilde\Omega \|_p \| \Theta - \tilde\Theta \|_p^{p-1}.
\end{aligned}
\end{align*}
This yields
\begin{align*}
\begin{aligned}
\frac{d}{dt}\| \Theta - \tilde\Theta \|_p \le \| \Omega - \tilde\Omega \|_p.
\end{aligned}
\end{align*}
On the other hand, we use the framework $({\mathcal F}_A)$ to find 
\begin{align*}
1-\frac{\kappa G'(\nu_c) C_{131}}{M_{{\mathcal Q}^{-1}[F]}} = 1-\frac{\kappa G'(\nu_c)^2 m_\Phi \cos\theta_*}{M_{{\mathcal Q}^{-1}[F]}} \leq 1-\frac{\kappa m_{G'}^2 m_\Phi \cos\theta_*}{M_{{\mathcal Q}^{-1}[F]}} \leq 0.
\end{align*}
Hence, we can obtain a Gr\"onwall type inequality
\begin{align*}
\begin{aligned} \label{C-7}
& \frac{d}{dt} \Big( \| \Theta(t) - \tilde\Theta(t) \|_p + \| \Omega(t) - \tilde\Omega(t) \|_p \Big) \\
&\hspace{1cm} \le \Lambda_{21} e^{-\Lambda_1 t} \| \Theta(t) - \tilde\Theta(t) \|_p +\Lambda_{22} e^{-\Lambda_1 t} \|\Omega(t) - \tilde\Omega(t) \|_p +\frac{\Lambda_{23}}{N^{\frac1p}} e^{-\Lambda_1t} \|\Phi - \tilde\Phi\|_p.
\end{aligned}
\end{align*}
Again, one has 
\begin{align*}
& \frac{d}{dt} \bigg[  \| \Theta(t) - \tilde\Theta(t) \|_p + \| \Omega(t) - \tilde\Omega(t) \|_p +\frac{\Lambda_{23}}{N^{\frac1p}\Lambda_1} e^{-\Lambda_1t} \|\Phi - \tilde\Phi\|_p \bigg ] \\
&\hspace{1cm} \le \Lambda_{21} e^{-\Lambda_1 t} \| \Theta(t) - \tilde\Theta(t) \|_p +\Lambda_{22} e^{-\Lambda_1 t} \|\Omega(t) - \tilde\Omega(t) \|_p \\
&\hspace{1cm} \le \max\{\Lambda_{21}, ~\Lambda_{22}\} e^{-\Lambda_1 t} \bigg[  \| \Theta(t) - \tilde\Theta(t) \|_p + \| \Omega(t) - \tilde\Omega(t) \|_p +\frac{\Lambda_{23}}{N^{\frac1p}\Lambda_1} e^{-\Lambda_1t} \|\Phi - \tilde\Phi\|_p \bigg].
\end{align*}
Then, Gr\"onwall's lemma yields
\begin{align*}
& \| \Theta(t) - \tilde\Theta(t) \|_p + \| \Omega(t) - \tilde\Omega(t) \|_p +\frac{\Lambda_{23}}{N^{\frac1p}\Lambda_1} e^{-\Lambda_1t} \|\Phi - \tilde\Phi\|_p \\
&\hspace{1cm} \leq \bigg( \| \Theta^0 - \tilde\Theta^0 \|_p + \| \Omega^0 - \tilde\Omega^0 \|_p +\frac{\Lambda_{23} \|\Phi - \tilde\Phi\|_p}{N^{\frac1p}\Lambda_1} \bigg) \exp\bigg(\frac{\max\{\Lambda_{21}, ~\Lambda_{22}\}}{\Lambda_1}\bigg). 
\end{align*}
Therefore, we can conclude our desired result with
\begin{align*}
\Lambda_2 := \max\bigg\{1, ~\frac{\Lambda_{23}}{\Lambda_1}\bigg\} \times \exp\bigg(\frac{\max\{\Lambda_{21}, ~\Lambda_{22}\}}{\Lambda_1}\bigg) > 0.
\end{align*}
This completes the proof of Theorem \ref{T3.1}. \newline

As a direct application of Theorem \ref{T3.1}, we obtain a uniform stability of phases. 
\begin{corollary} \label{C3.1}
Suppose that the framework $({\mathcal F}_A)$ holds, and let $\Theta = \Theta(t)$ and $\tilde\Theta = \tilde\Theta(t)$ be two global solutions of \eqref{sys:GKM} subject to the initial data, natural frequency, and communication matrix triples $(\Theta^0, \cV, \Phi)$ and $(\tilde\Theta^0, \tilde\cV, \tilde\Phi)$, respectively. Then, there exists a positive constant $\Lambda_3 = \Lambda_3(\kappa, F, a_\Phi, b_\Phi, a_G, b_G, \theta_*)$ such that for $p\in[1,\infty]$,
\[
\sup_{0 \leq t < \infty} \| \Theta(t) - \tilde\Theta(t) \|_p \leq \Lambda_3 \bigg( \| \Theta^0 - \tilde\Theta^0 \|_p + \|\mathcal V - \tilde{\mathcal V}\|_p +\frac{\|\Phi - \tilde\Phi\|_p}{N^{\frac1p}} \bigg).
\]
\end{corollary}
\begin{proof}
As in the proof of Theorem \ref{T3.1}, we only consider the case of $p\in(1, \infty)$. Similar to the proof of Theorem \ref{T3.1}, one has 
\begin{align*}
\begin{aligned}
|\omega_\alpha -\tilde\omega_\alpha| &= \bigg| G\bigg(\nu_\alpha + \frac{\kappa}{N}\sum_{\beta=1}^N \phi_{\alpha\beta} \sin(\theta_\beta - \theta_\alpha) \bigg) - G\bigg(\tilde\nu_\alpha + \frac{\kappa}{N}\sum_{\beta=1}^N \tilde\phi_{\alpha\beta} \sin(\tilde\theta_\beta - \tilde\theta_\alpha) \bigg) \bigg| \\
&\leq M_{G'}|\nu_\alpha -\tilde\nu_\alpha| +\frac{\kappa M_{G'}}{N} \sum_{\beta=1}^N \big| \phi_{\alpha\beta} \sin(\theta_\beta - \theta_\alpha) -\tilde\phi_{\alpha\beta} \sin(\tilde\theta_\beta - \tilde\theta_\alpha) \big| \\
&\leq M_{G'}|\nu_\alpha -\tilde\nu_\alpha| +\frac{\kappa M_{G'}}{N} \sum_{\beta=1}^N \big| \phi_{\alpha\beta} -\tilde\phi_{\alpha\beta} \big|  \\
&\hspace{.2cm}+\frac{\kappa b_\Phi M_{G'}}{N} \sum_{\beta=1}^N \big| \sin(\theta_\beta - \theta_\alpha) -\sin(\tilde\theta_\beta - \tilde\theta_\alpha) \big| \\
&\leq M_{G'}|\nu_\alpha -\tilde\nu_\alpha| +\frac{\kappa M_{G'}}{N^{\frac1p}} \bigg( \sum_{\beta=1}^N \big| \phi_{\alpha\beta} -\tilde\phi_{\alpha\beta} \big|^p \bigg)^{\frac1p} +\kappa b_\Phi M_{G'} \big| \theta_\alpha -\tilde\theta_\alpha \big| \\
&\hspace{.2cm}+\frac{\kappa b_\Phi M_{G'}}{N^{\frac1p}} \bigg( \sum_{\beta=1}^N \big| \theta_\beta -\tilde\theta_\beta \big|^p \bigg)^{\frac1p}.
\end{aligned}
\end{align*}
Then, the Minkowski inequality implies
\begin{equation} \label{C-8}
\|\Omega-\tilde\Omega\|_p \leq M_{G'} \|\mathcal V -\tilde{\mathcal V} \|_p +2\kappa b_\Phi M_{G'} \|\Theta -\tilde\Theta\|_p +\frac{\kappa M_{G'}}{N^{\frac1p}} \|\Phi-\tilde\Phi\|_p.
\end{equation}
Now, we combine Theorem \ref{T3.1} and \eqref{C-8} to obtain
\begin{align*}
\| \Theta(t) - \tilde\Theta(t) \|_p &\leq \| \Theta(t) - \tilde\Theta(t) \|_p + \| \Omega(t) -\tilde\Omega(t) \|_p \\
&\leq \Lambda_2 \bigg( \| \Theta^0 - \tilde\Theta^0 \|_p + \| \Omega^0 -\tilde\Omega^0\|_p +\frac{\|\Phi - \tilde\Phi\|_p}{N^{\frac1p}} \bigg) \\
&\leq \Lambda_2 \bigg( \big(1+2\kappa b_\Phi M_{G'} \big) \|\Theta^0 -\tilde\Theta^0\|_p +M_{G'} \|\mathcal V -\tilde{\mathcal V} \|_p +\frac{1+\kappa M_{G'}}{N^{\frac1p}} \|\Phi-\tilde\Phi\|_p  \bigg).
\end{align*}
Therefore, we can conclude our desired result with
\begin{align*}
\Lambda_3 := \Lambda_2 \times \max\big\{ M_{G'}, 1+\kappa M_{G'} , 1+2\kappa b_\Phi M_{G'} \big\} > 0.
\end{align*}
\end{proof}

%%%%%%%%%%%%%%%%%%%%%%%%%%%%%
%
%
% Section 4
%
%
%%%%%%%%%%%%%%%%%%%%%%%%%%%%%
\section{The continuum general Kuramoto model} \label{sec:4}
\setcounter{equation}{0}
In this section, we study a global well-posedness of the continuum GK model \eqref{sys:CGKM}, uniform-in-time continuum limit of \eqref{sys:LGKM} and the uniform $L^{\infty}$-contraction property of the continuum GK model \eqref{sys:CGKM}.
\subsection{Gallery of main results} \label{sec:4.1}
In this subsection, we briefly summarize a set of main results for the continuum GK model without proofs.  First, we recall the concept of ``classical solution'' to the continuum GK model \eqref{sys:CGKM}. \newline
\begin{definition} \label{D4.1}
For $T\in(0, \infty]$, let $\theta = \theta(t, x)$ be a classical solution to the Cauchy problem \eqref{sys:CGKM} in the time-interval $[0, T]$ $([0, \infty)$ when $T=\infty)$, if the following conditions hold. \newline
\begin{enumerate}
\item The phase function $\theta$ satisfies the following regularity condition:
\[
\theta \in \mathcal C^1 \big( [0, T]; L^\infty(D) \big).
\]
\vspace{0.1cm}

\item The phase function $\theta$ satisfies the integro-differential equation pointwise:
\begin{align*}
F(\partial_t \theta(t,x)) = \nu(x) + \frac{\kappa}{\mu(D)}\int_D \phi(x,z)\sin(\theta(t, z)-\theta(t, x)) d\mu(z), 
\end{align*}
for all $t\in[0, T]$ and a.e. $x\in D.$

\vspace{0.1cm}

\item 
The phase function $\theta$ satisfies initial datum pointwise:
\[ \theta(0, x) = \theta^0(x), \quad \mbox{for a.e. $x \in D$}. \]
\end{enumerate}
\end{definition}

\vspace{0.5cm}

Recall that throughout the paper, $D\subset\bbr^d$ is a compact region with positive measure. In what follows, we we present a global unique solvability of a classical solution to the continuum GK model \eqref{sys:CGKM}.

\begin{proposition} \label{P4.1}
Suppose that system functions and initial data satisfy
\begin{align*}
\|\theta^0\|_\infty < \infty, \quad 0 < \|\nu\|_\infty +\|\phi\|_{L^\infty_xL^1_y} < \infty
\end{align*}
Then, there exists a unique global classical solution $\theta(t, x) \in \mathcal C^1 \big( [0, \infty); L^\infty(D) \big)$ to the continuum GK model \eqref{sys:CGKM}. Moreover, if we further assume that
\begin{align*}
\theta^0, \nu \in \mathcal C(D), \quad \phi \in \mathcal C(D^2),
\end{align*}
then we have $\theta(t, x) \in \mathcal C^1 \big( [0, \infty); \mathcal C(D) \big)$.
\end{proposition}
\begin{proof}
Since this can be treated in a canonical way, we leave it in Appendix \ref{App-B}.
\end{proof}
Second, we recall the concept of continuum limit from the lattice GK model \eqref{sys:LGKM} toward the continuum GK model \eqref{sys:CGKM}.
\begin{definition} \label{D4.2}
\cite{H-K-M1, H-K-M2, H-K-P-R-S2} For given $T \in (0,\infty]$, the continuum model \eqref{sys:CGKM} is “derivable” from the lattice model \eqref{sys:LGKM}  in $[0,T]$ $([0, \infty)$ when $T=\infty)$, if a solution to \eqref{sys:CGKM} can be obtained as a suitable limit of a sequence of lattice solutions to \eqref{sys:LGKM}, as the lattice spacing tends to zero.
\end{definition}

\vspace{0.2cm}

\noindent Next, we state a second sufficient framework $({\mathcal F}_B)$ leading to the uniform-in-time continuum limit:
\begin{itemize}
\item
$({\mathcal F}_B1)$:~$D$ satisfies a partition property $(\mathcal P)$, i.e., there exists a sequence of partition $\{\mathcal P^N\}_{N\geq1}$ such that
\begin{align*}\label{D-0}
\begin{aligned}
& \mathcal P^N = \{D^N_\alpha\}_{\alpha=1}^{2^N}, \quad \mathcal P^0 = \{D^0_1\} = \{D\}, \quad \lim_{N\to\infty} \max_{1\leq\alpha\leq2^N} \mathrm{diam}\big(D^N_\alpha\big) = 0, \\
& D^N_\alpha = D^{N+1}_{2\alpha-1} \uplus D^{N+1}_{2\alpha}, \quad \mu \big( D^{N+1}_{2\alpha-1} \big) = \mu \big( D^{N+1}_{2\alpha} \big), \quad 1\leq \alpha \leq 2^N.
\end{aligned}
\end{align*}
Note that $n$-cube, $n\geq1$ is a trivial example satisfying the above condition, and one can easily observe that $\mu(D^N_\alpha) = \mu(D)/2^N$ for all $1\leq \alpha \leq 2^N$.\newline
\vspace{0.1cm}
\item
$({\mathcal F}_B2)$:~The system functions $\nu,~\phi$ and initial datum $\theta^0$ are continuous. \newline
\vspace{0.1cm}
\item
$({\mathcal F}_B3)$:~system parameters satisfy the following set of conditions, i.e.,~there exist $b_\Phi \geq a_\Phi>0$, $-L < a_G < b_G < L$, and $\theta_*\in(0, \pi/2)$ such that
\begin{equation*} \label{D-1}
\begin{cases}
\displaystyle  \cD(\theta^0) \leq \theta_*, \quad a_\Phi \leq \phi(x, y) \leq b_\Phi, \quad \forall~x, y\in D, \\
\displaystyle \kappa  > \max \Big \{ \frac{\cD(\nu)}{ a_\Phi \sin\theta_*},~\frac{2M_{{\mathcal Q}^{-1}[F]}}{ a_\Phi m_{G'}^2 \cos\theta_*}    \Big \} > 0, \\
\displaystyle  F(a_G)+\kappa b_\Phi \leq \nu(x) \leq F(b_G)-\kappa b_\Phi, \quad \forall~x \in D,
\end{cases}
\end{equation*}
where $M_{{\mathcal Q}^{-1}[F]}$ is defined in \eqref{C-1} and \eqref{C-2}.
\end{itemize}

\vspace{0.5cm}

For given triple $(\nu, \phi,\theta^0)$, let $\{ \theta^N_\alpha \}_{\alpha=1}^{2^N} $ be a global smooth solution of the following Cauchy problem:
\begin{align*}
\begin{dcases}
F(\dot{\theta}_\alpha) = \nu^N_\alpha + \displaystyle\frac{\kappa}{2^{N}}\sum_{\beta=1}^{2^{N}} \phi^N_{\alpha\beta}\sin(\theta_\beta - \theta_\alpha),\quad \forall~t>0, \\
\theta_\alpha \big|_{t=0+} = \theta^{N,0}_\alpha, \quad \forall~\alpha=1,\cdots,2^{N},
\end{dcases}
\end{align*}
where system parameters and initial data are defined as
\begin{align*} \label{D-2}
\begin{aligned} 
& \nu^N_\alpha := \frac{2^N}{\mu(D)} \int_{D^N_\alpha} \nu(x) d\mu(x) ,\quad \theta^{N,0}_\alpha := \frac{2^N}{\mu(D)} \int_{D^N_\alpha} \theta^0(x)d\mu(x), \\
& \phi^N_{\alpha\beta} := \frac{2^{2N}}{\mu(D)^2} \iint_{D^N_\alpha \times D^N_\beta}  \phi(x,y) d\mu(x)d\mu(y), \quad \forall~\alpha, \beta = 1,\cdots,2^N, \quad \forall~N\geq0.
\end{aligned}
\end{align*}
By Lemma \ref{L4.1} in below and direct calculation, one can see that 
\begin{equation} \label{D-3}
\theta^N(t,x) :=\sum_{\alpha=1}^{2^N} \theta^N_\alpha(t) \mathbf{1}_{D^N_\alpha}(x)
\end{equation}
is a solution to the following Cauchy problem:
\begin{align*}
\begin{cases}
\displaystyle F(\partial_t \theta(t,x)) = \nu^N(x) +\frac{\kappa}{\mu(D)}\int_D \phi^N(x,z)\sin(\theta(t,z)-\theta(t,x))dz, \quad t>0, \vspace{.1cm}\\
\theta(0,x) = \theta^{N,0}(x),
\end{cases}
\end{align*}
where system functions and initial data are defined as
\begin{align} \label{D-4}
\nu^N := \sum_{\alpha=1}^{2^N} \nu^N_\alpha \mathbf{1}_{D^N_\alpha},\quad \theta^{N,0} :=  \sum_{\alpha=1}^{2^N} \theta^{N,0}_\alpha \mathbf{1}_{D^N_\alpha}, \quad \phi^N :=  \sum_{\alpha,\beta=1}^{2^N} \phi^N_{\alpha\beta} \mathbf{1}_{D^N_\alpha} \otimes \mathbf{1}_{D^N_\beta}.
\end{align}
Then, we will show the convergence of $\theta^N$ to justify the continuum limit procedure. 
\begin{theorem} \label{T4.1}
Suppose the framework $({\mathcal F}_B1)$-$({\mathcal F}_B3)$ hold, and let $\{\theta^N\}_{N\geq1}$ be a sequence of simple functions defined in \eqref{D-3}. Then, there exists unique limit $\theta\in\mathcal C^1 \big( [0, \infty); L^\infty (D) \big)$ satisfying the following two assertions:
\begin{enumerate}[(1)]
\item $\displaystyle \lim_{N\to\infty} \sup_{t\geq0} \underset{x\in D}{\mathrm{ess} \sup} |\theta^N(t, x) -\theta(t, x)| = 0.$
\vspace{0.1cm}
\item $\theta$ admits the continuous version which is the unique classical solution to the continuum GK model \eqref{sys:CGKM}.
\end{enumerate}
\end{theorem}
\begin{proof} We leave its proof in Section \ref{sec:4.2}. 
\end{proof}
\begin{remark}
One can obtain finite-in-time continuum limit under more generous setting using Lebesgue differentiation theorem (see Appendix \ref{App-C}).
\end{remark}
Third, we present the $L^\infty$-contraction of classical solutions obtained in Proposition \ref{P4.1} which corresponds to the continuum analogue of Corollary \ref{C3.1}. 
\begin{proposition} \label{P4.2}
Suppose the framework $({\mathcal F}_B1)$-$({\mathcal F}_B3)$ hold, and we also assume that there exist $a_\phi>0$ and $\theta_*\in(0, \pi/2)$ such that initial data and system functions satisfy
\[
\phi(x,y) \ge a_\phi>0,\quad 0<\mathcal{D}(\theta^0) ,\mathcal{D}(\tilde\theta^0) <\theta_*, \quad
\kappa  > \frac{\cD(\nu)}{a_\phi \sin \theta_*}, \quad
\|\theta^0 - \tilde \theta^0 \|_\infty  < \pi.
\]
Let $\theta, \tilde\theta \in \mathcal{C}^1([0,\infty); \mathcal{C}(D))$ be a global classical solution to \eqref{sys:CGKM} subject to continuous initial data $\theta^0$ and $\tilde\theta^0$, respectively. Then, we have $L^{\infty}$-contraction:
\[  \| \theta(t,\cdot) - \tilde\theta(t,\cdot) \|_\infty + \int_0^t {\mathcal P}(s) ds  \le\| \theta^0 - \tilde\theta^0 \|_\infty, \]
where ${\mathcal P}$ is the production rate functional defined by 
\[ {\mathcal P}(s):= \frac{\kappa m_{G'} a_\phi \sin2\theta_*}{2\theta_* \mu(D)} \Big( \mu(D)\| \theta(s,\cdot) - \tilde\theta(s,\cdot)\|_\infty - \|\theta(s,\cdot) - \tilde\theta(s,\cdot)\|_{L^1} \Big) \ge 0. \]
\end{proposition}
\begin{proof} We leave its proof in Section \ref{sec:4.3}
\end{proof}
In the following two subsections, we study proofs of Theorem \ref{T4.1} and Proposition \ref{P4.2}.
\subsection{Uniform-in-time continuum limit} \label{sec:4.2} 
In next lemma, we show that the piecewise constancy will propagate along the dynamics  \eqref{sys:CGKM}.
\begin{lemma} \label{L4.1}
Suppose that the communication weight function, natural frequency function, and initial data are simple functions in the sense that
\begin{align*}
D = \biguplus_{\alpha\in\Lambda} D_\alpha, \hspace{.2cm} \phi(x,y)=\phi_{\alpha\beta}, \hspace{.2cm} \nu(x)=\nu_\alpha, \hspace{.2cm} \theta^0(x) = \theta_\alpha^{0}, \hspace{.2cm} \forall~(x,y)\in D_\alpha \times D_\beta, \hspace{.2cm} \forall~\alpha, \beta\in\Lambda,
\end{align*}
where $\Lambda$ is a finite index set, and let $\theta=\theta(t, x)$ be a global solution to the continuum GK model \eqref{sys:CGKM}. Then the phase function $\theta$ is piecewise constant in the sense that
$$\theta(t, x)=\theta(t, z), \quad \forall~x, z \in D_\alpha, \quad \forall~t\ge 0,\quad \forall~\alpha\in\Lambda.$$
\end{lemma}
\begin{proof}
For fixed $x,z\in D_\alpha$, we define
\[
\eta(t, x,z):=\theta(t, x)-\theta(t, z), \quad {\mathcal Q}[F](x, y) :=
\begin{cases}
\displaystyle \frac{F(x)-F(y)}{x-y}, & x\ne y, \\
F'(x), & x=y.
\end{cases}
\]
Then, we have
\begin{align*}
\begin{aligned}
& {\mathcal Q}[F](\partial_t\theta(t, x), \partial_t\theta(t, z)) \partial_t\eta(t, x, z) \\
&\hspace{.6cm} =\nu(x)-\nu(z)+\frac{\kappa}{\mu(D)} \int_D \big[ \phi(x,y)\sin(\theta(t, y)-\theta(t, x)) -\phi(z,y)\sin(\theta(t, y)-\theta(t, z)) \big] dy \\
&\hspace{.6cm} = \frac{\kappa}{\mu(D)} \sum_{\beta\in\Lambda} \phi_{\alpha\beta} \int_{D_\beta} \big[ \sin(\theta(t, y)-\theta(t, x)) -\sin(\theta(t, y)-\theta(t, z)) \big] dy \\
&\hspace{.6cm} = -\frac{2\kappa}{\mu(D)} \sin\bigg( \frac{\eta(t, x, z)}{2} \bigg) \sum_{\beta\in\Lambda} \phi_{\alpha\beta} \int_{D_\beta} \cos\bigg( \theta(t, y)-\frac{\theta(t, x)+\theta(t, z)}{2} \bigg) dy \\
&\hspace{.6cm} = -\frac{\kappa \eta(t, x, z)}{\mu(D)} A\bigg( \frac{\eta(t, x, z)}{2} \bigg) \sum_{\beta\in\Lambda} \phi_{\alpha\beta} \int_{D_\beta} \cos\bigg( \theta(t, y)-\frac{\theta(t, x)+\theta(t, z)}{2} \bigg) dy,
\end{aligned}
\end{align*}
where the analytic function $A$ is defined by
\begin{align*}
A(x) :=
\begin{dcases}
\frac{\sin x}{x}, & x\ne0, \\
1, & x=0.
\end{dcases}
\end{align*}
Therefore, the fact that $\eta(0, x, z) = 0$ for all $x, z\in D_\alpha$ and the uniqueness of the solution implies
\begin{align*}
\eta(t, x, z) = 0, \quad \forall~x, z\in D_\alpha, \quad \forall~t\geq0,
\end{align*}
which gives the desired result.
\end{proof}
\begin{lemma} \label{L4.2}
Suppose the framework $({\mathcal F}_B1) - ({\mathcal F}_B3)$ hold, and let $(\nu^N)$, $(\theta^{N, 0})$, and $(\phi^N)$ be sequences of functions constructed in \eqref{D-3} and \eqref{D-4}. Then, one has the following uniform convergence:
$$
\lim_{N\to\infty} \Big( \|\theta^{N,0} - \theta^0\|_\infty +\|\nu^N-\nu\|_\infty +\|\phi^N-\phi\|_\infty \Big) = 0.
$$
\end{lemma}
\begin{proof}
It is enough to show that
\begin{align*}
\lim_{N\to\infty} \|\theta^{N,0} - \theta^0\|_\infty = 0.
\end{align*}
The other cases can be treated similarly. Let $\varepsilon$ be any positive constant. Since $\theta^0$ is uniformly continuous on $D$, there exists $\delta>0$ such that 
\[ |x-y|<\delta \quad \Longrightarrow \quad  |\theta^0(x) -\theta^0(y)| <\varepsilon, \]
Furthermore, we have $N^*\gg1$ such that
\begin{align*}
\mathrm{diam}\big(D^N_\alpha\big) \leq \delta, \quad 1\leq\alpha\leq2^N, \quad N\geq N^*.
\end{align*}
For $N\geq N^*$ and $x\in D^N_\alpha$, 
\begin{align*}
\big| \theta^{N, 0}(x) -\theta^0(x) \big| = \big| \theta^{N, 0}_\alpha -\theta^0(x) \big| = \bigg| \frac{2^N}{\mu(D)} \int_{D^N_\alpha} \big( \theta^0(y) -\theta^0(x) \big) d\mu(y) \bigg| < \varepsilon.
\end{align*}
This implies our desired result.
\end{proof}
Now, we are ready to provide a proof of Theorem \ref{T4.1}. Since the proof is very lengthy, we split its proof into several steps. \newline

\noindent $\bullet$~Step A (Construction of approximate solutions):~Fix $N_2 > N_1\geq1$ and note that
\begin{align*}
& D^{N_1}_\alpha = \biguplus_{\beta=1}^{2^{N_2-N_1}} D^{N_2}_{1-\beta+\alpha\times2^{N_2-N_1}} \\
& \qquad \implies \theta^{N_2}(t, x) -\theta^{N_1}(t, x) = \sum_{\alpha=1}^{2^{N_2}} \big[ \theta^{N_2}_\alpha(t) -\theta^{N_1}_{\lceil\alpha\times2^{N_1-N_2}\rceil}(t) \big] {\bf 1}_{D^{N_2}_\alpha}(x),
\end{align*}
where $\lceil x\rceil$ means the least integer greater than or equal to $x$. Hence, we define some copies of $\Theta^{N_1}$, $\cV^{N_1}$, and $\Phi^{N_1}$ as
\begin{align*}
& \theta^{N_1, N_2}_\alpha(t) := \theta^{N_1}_{\lceil\alpha\times2^{N_1-N_2}\rceil}(t), \quad \theta^{N_1, N_2, 0}_\alpha := \theta^{N_1, 0}_{\lceil\alpha\times2^{N_1-N_2}\rceil}, \quad \nu^{N_1, N_2}_\alpha := \nu^{N_1}_{\lceil\alpha\times2^{N_1-N_2}\rceil}, \\
& \phi^{N_1, N_2}_{\alpha\beta} := \phi^{N_1}_{\lceil\alpha\times2^{N_1-N_2}\rceil \lceil\beta\times2^{N_1-N_2}\rceil}, \quad \forall ~\alpha, \beta = 1, \cdots, 2^{N_2},
\end{align*}
and note that 
\begin{align*}
F\big( \dot\theta^{N_1, N_2}_\alpha \big) &= F\big( \dot\theta^{N_1}_{\lceil\alpha\times2^{N_1-N_2}\rceil} \big) \\
&= \nu^{N_1}_{\lceil\alpha\times2^{N_1-N_2}\rceil} +\frac{\kappa}{2^{N_1}}\sum_{\beta=1}^{2^{N_1}} \phi^{N_1}_{{\lceil\alpha\times2^{N_1-N_2}\rceil}\beta}\sin\big( \theta^{N_1}_\beta - \theta^{N_1}_{\lceil\alpha\times2^{N_1-N_2}\rceil} \big) \\
&= \nu^{N_1, N_2}_\alpha +\frac{\kappa}{2^{N_2}}\sum_{\beta=1}^{2^{N_1}} 2^{N_2-N_1} \phi^{N_1}_{{\lceil\alpha\times2^{N_1-N_2}\rceil}\beta}\sin\big( \theta^{N_1}_\beta - \theta^{N_1, N_2}_\alpha \big) \\
&= \nu^{N_1, N_2}_\alpha + \displaystyle\frac{\kappa}{2^{N_2}}\sum_{\beta=1}^{2^{N_1}} \sum_{\gamma=(\beta-1)2^{N_2-N_1}+1}^{\beta\times2^{N_2-N_1}} \phi^{N_1, N_2}_{\alpha\gamma} \sin\big( \theta^{N_1, N_2}_\gamma - \theta^{N_1, N_2}_\alpha \big) \\
&= \nu^{N_1, N_2}_\alpha + \displaystyle\frac{\kappa}{2^{N_2}} \sum_{\gamma=1}^{2^{N_2}} \phi^{N_1, N_2}_{\alpha\gamma} \sin\big( \theta^{N_1, N_2}_\gamma - \theta^{N_1, N_2}_\alpha \big).
\end{align*}
Thus, $\Theta^{N_1, N_2}$ is a solution to the following Cauchy problem:
\begin{align*}
\begin{dcases}
F(\dot{\theta}_\alpha) = \nu^{N_1, N_2}_\alpha + \displaystyle\frac{\kappa}{2^{N_2}}\sum_{\beta=1}^{2^{N_2}} \phi^{N_1, N_2}_{\alpha\beta}\sin(\theta_\beta - \theta_\alpha),\quad \forall~t>0, \\
\theta_\alpha \big|_{t=0+} = \theta^{N_1, N_2, 0}_\alpha, \quad \forall~\alpha=1,\cdots,2^{N_2}.
\end{dcases}
\end{align*}

\vspace{0.2cm}

\noindent $\bullet$~Step B (Extracting a limit function):~It follows from
\begin{align*}
\sum_{\alpha=1}^{2^{N_2}} \nu^{N_1, N_2}_\alpha = \sum_{\alpha=1}^{2^{N_2}} \nu^{N_1}_{\lceil\alpha\times2^{N_1-N_2}\rceil} = 2^{N_2-N_1} \sum_{\alpha=1}^{2^{N_1}} \nu^{N_1}_\alpha = \frac{2^{N_2}}{\mu(D)} \int_D \nu(x) dx = \sum_{\alpha=1}^{2^{N_2}} \nu^{N_2}_\alpha,
\end{align*}
and assumptions that we can apply Corollary \ref{C3.1} to $\Theta^{N_1, N_2}$ and $\Theta^{N_2}$ to see
\begin{align} 
\begin{aligned} \label{D-5}
& \big\| \theta^{N_2}(t, \cdot) -\theta^{N_1}(t, \cdot) \big\|_\infty \\
&\hspace{1cm} = \max_{1 \leq \alpha \leq 2^{N_2}} \big| \theta^{N_2}_\alpha(t) -\theta^{N_1, N_2}_\alpha (t) \big| = \big\| \Theta^{N_2}(t) -\Theta^{N_1, N_2}(t) \big\|_\infty \\
&\hspace{1cm} \lesssim \big\| \theta^{N_2}(0) -\theta^{N_1, N_2}(0) \big\|_\infty +\big\| \mathcal V^{N_2} -\mathcal V^{N_1, N_2} \big\|_\infty +\big\| \Phi^{N_2} -\Phi^{N_1, N_2} \big\|_\infty \\
&\hspace{1cm} = \big\| \theta^{N_2, 0} -\theta^{N_1, 0} \big\|_\infty +\big\| \nu^{N_2} -\nu^{N_1} \big\|_\infty +\big\| \phi^{N_2} -\phi^{N_1} \big\|_\infty.
\end{aligned}
\end{align} 
By Lemma \ref{L4.2}, the R.H.S. of \eqref{D-5} tends to zero as $N_1, N_2 \to \infty$. Thus, $\big\{ \theta^{N}(t, \cdot) \big\}_{N=1}^\infty$ is a Cauchy sequence in $L^\infty(D)$, and there exists a pointwise limit $\theta^*(t, \cdot)\in L^\infty(D)$ such that 
\[ 
\theta^{N}(t, \cdot) \to \theta^*(t, \cdot), \quad \mbox{as $N \to \infty$}. 
\]
\noindent $\bullet$~Step C (The limit function satisfies the first assertion):~Especially, we can say $\theta^*(0, \cdot) = \theta^0$ by taking the continuous version. In \eqref{D-5}, we let $N_2 \to \infty$ to obtain
\begin{equation*} \label{D-6}
\underset{x\in D}{\mathrm{ess} \sup} \big| \theta^*(t, x) -\theta^{N_1}(t, x) \big| \lesssim \Big( \big\| \theta^0 -\theta^{N_1, 0} \big\|_\infty +\big\| \nu -\nu^{N_1} \big\|_\infty +\big\| \phi -\phi^{N_1} \big\|_\infty \Big).
\end{equation*}
Again, we take $N_1 \to \infty$ and use Lemma \ref{L4.2} to find the first assertion.

\vspace{0.2cm}

\noindent $\bullet$~Step D (The limit function is a classical solution):  We check defining relations in Definition \ref{D4.1} one by one. \newline

\noindent $\diamond$~Step D.1 (Regularity of the limit function):~ for arbitrary $\varepsilon>0$, we fix $N\gg1$ such that
\begin{align*}
\big\| \theta^0 -\theta^{N, 0} \big\|_\infty +\big\| \nu -\nu^N \big\|_\infty +\big\| \phi -\phi^N \big\|_\infty < \varepsilon,
\end{align*}
and find $\delta>0$ such that
\begin{align*}
\big\| \theta^N(t+h, \cdot) -\theta^N(t, \cdot) \big\|_\infty < \varepsilon, \quad \forall~h\in(-\delta, \delta).
\end{align*}
We use \eqref{D-5} to see
\begin{align*}
\begin{aligned}
& \underset{x\in D}{\mathrm{ess} \sup} \big| \theta^*(t+h, x) -\theta^*(t, x) \big| \\
&\hspace{1cm} \leq \underset{x\in D}{\mathrm{ess} \sup} \big| \theta^*(t+h, x) -\theta^N(t+h, x) \big| \\
&\hspace{1,2cm} +\big\| \theta^N(t+h, \cdot) -\theta^N(t, \cdot) \big\|_\infty +\underset{x\in D}{\mathrm{ess} \sup} \big| \theta^N(t, x) -\theta^*(t, x) \big| \\
&\hspace{1cm} \lesssim 2\big\| \theta^0 -\theta^{N, 0} \big\|_\infty +2\big\| \nu -\nu^N \big\|_\infty +2\big\| \phi -\phi^N \big\|_\infty +\varepsilon < 3\varepsilon, \quad \forall~h\in(-\delta, \delta),
\end{aligned}
\end{align*}
so that we can conclude that $\theta^*\in\mathcal C \big( [0, \infty); L^\infty(D) \big)$. \newline

\noindent $\diamond$~Step D.2 (The limit function satisfies the integro-differential equation):~we claim that $\theta^*$ satisfies
\begin{align} \label{D-7}
\theta^*(t,x) -\theta^{0}(x) = \int_0^t G\bigg( \nu(x) +\frac{\kappa}{\mu(D)}\int_D \phi(x,z)\sin\big( \theta^*(z,s) -\theta^*(x,s) \big)dz \bigg) ds,
\end{align}
for all $t\geq0$ and almost every $x\in D$. Note that once we have \eqref{D-7}, then we have $\theta^*\in\mathcal C^1 \big( [0, \infty); L^\infty(D) \big)$ and moreover, the uniqueness given in Proposition \ref{P4.1} enables us to choose the continuous version $\theta^*\in\mathcal C^1 \big( [0, \infty); C(D) \big)$.  \newline

\noindent {\it Proof of \eqref{D-7}}:~For fixed $t\geq0$, we have
\begin{align}
\begin{aligned} \label{D-8}
& \theta^N(t,x) -\theta^{N, 0}(x) \\
& \hspace{1cm} = \int_0^t G\bigg( \nu^N(x) +\frac{\kappa}{\mu(D)}\int_D \phi^N(x,z)\sin\big( \theta^N(z,s) -\theta^N(x,s) \big)dz \bigg) ds.
\end{aligned}
\end{align}
It is easy to check that the LHS of \eqref{D-8} converges to $\theta^*(t,x) -\theta^{0}(x)$ for almost every $x\in D$ as $N$ goes to infinity. Moreover, it follows from
\begin{align*}
& \bigg\| \int_0^t \bigg[ G\bigg( \nu^N +\frac{\kappa}{\mu(D)}\int_D \phi^N(\cdot,z)\sin\big( \theta^N(z,s) -\theta^N(\cdot,s) \big)dz \bigg) \\
&\hspace{1cm} -G\bigg( \nu +\frac{\kappa}{\mu(D)}\int_D \phi(\cdot,z)\sin\big( \theta^*(z,s) -\theta^*(\cdot,s) \big)dz \bigg) \bigg] ds \bigg\|_\infty \\
&\hspace{1cm} \leq tM_{G'} \bigg( \big\| \nu -\nu^N \big\|_\infty +\kappa\big\| \phi -\phi^N \big\|_\infty +2\kappa b_\Phi \sup_{0\leq s\leq t} \underset{x\in D}{\mathrm{ess} \sup} \big| \theta^*(s, x) -\theta^N(s, x) \big| \bigg),
\end{align*}
Lemma \ref{L4.2} and Step C  that the RHS of \eqref{D-8} converges to
\begin{align*}
\int_0^t G\bigg( \nu(x) +\frac{\kappa}{\mu(D)}\int_D \phi(x,z)\sin\big( \theta^*(z,s) -\theta^*(x,s) \big)dz \bigg) ds
\end{align*}
for almost every $x\in D$ as $N$ goes to infinity. This complete the verification of \eqref{D-8}.
\end{proof}
\subsection{$L^\infty$-contraction} \label{sec:4.3}
In this subsection, we present the continuum analogue of Lemma \ref{L2.1} and apply this to prove the $L^\infty$-contraction property of the solution operator to the continuum GK model \eqref{sys:CGKM}.
\begin{lemma} \label{L4.3}
Suppose that there exist $a_\phi>0$ and $\theta_*\in(0, \pi/2)$ such that initial data and system functions satisfy
\begin{gather*}
\phi(x,y) \ge a_\phi, \quad 0< \mathcal{D}(\theta^0) < \theta_*,\quad \kappa  > \frac{\cD(\nu)}{a_\phi \sin \theta_*},
\end{gather*}
and let $\theta\in \mathcal{C}^1([0,\infty); \mathcal{C}(D))$ be a global classical solution to \eqref{sys:CGKM} whose existence is guaranteed by Proposition \ref{P4.1}. Then $\theta_*$ is an upper bound for the diameter $\mathcal{D}(\theta(t))$:
\[ \mathcal{D}(\theta(t, \cdot)) \le \theta_*,\quad\forall ~t > 0. \]
\end{lemma}
\begin{proof}
We take $\varepsilon$ that
$$0<\varepsilon<\frac{\pi}{2} - \theta_*.$$
We assume that
$$ \mathcal{T}_\varepsilon := \{ t>0\,:\,\mathcal D(\theta(t)) > \theta_* + \varepsilon \}\neq \emptyset.$$
By the continuity of $\mathcal D(\theta(t))$, we may set
$$ t_* := \inf \mathcal{T}_\varepsilon > 0,\quad \mathcal D(\theta(t_*)) = \theta_* + \varepsilon. $$
We fix a pair $(x,y)\in D^2$ so that
$$ \mathcal D(\theta(t_*)) = \theta(t_*,x) -\theta(t_*,y).$$
Now, we have some point $P_{x,y}$ as a result of mean value theorem such that
\begin{align*}
\begin{aligned}
&\frac{d}{dt}\bigg|_{t=t_*} \big( \theta(t,x) - \theta(t,y) \big) \\
&=
G\bigg( \nu(x) + \frac{\kappa}{\mu(D)}\int_D \phi(x,z) \sin(\theta(t_*,z) - \theta(t_*,x)) dz \bigg) \\
&\hspace{.2cm}- G\bigg( \nu(y) + \frac{\kappa}{\mu(D)}\int_D \phi(y,z) \sin(\theta(t_*,z) - \theta(t_*,y)) dz \bigg)\\
&\le G'(P_{x,y})\bigg(\cD(\nu) -  \frac{\kappa}{\mu(D)} \int_D \bigg( \phi(x,z) \sin(\theta(t_*,x) - \theta(t_*,z))  \\
&\hspace{5.2cm}+  \phi(y,z) \sin(\theta(t_*,z) - \theta(t_*,y))   \bigg)dz
\bigg)\\
&\le G'(P_{x,y})\bigg(\cD(\nu)  -  \frac{\kappa a_\phi}{\mu(D)} \int_D \bigg(  \sin(\theta(t_*,x) - \theta(t_*,z)) +  \sin(\theta(t_*,z) - \theta(t_*,y))   \bigg)dz
\bigg)\\
&\le G'(P_{x,y})\bigg(\cD(\nu)  - \kappa a_\phi \sin\big(\theta(t_*,x) - \theta(t_*,y) \big)\bigg) \\
&< G'(P_{x,y})\bigg(\cD(\nu)  - \kappa a_\phi \sin\theta_* \bigg) < 0.
\end{aligned}
\end{align*}
However, this implies the existence of $\delta\in(0,t_*)$ satisfying
$$\mathcal{D}_\theta(t_* - \delta) \ge \theta(t_* - \delta,x)-\theta(t_* - \delta,y) > \theta(t_* ,x)- \theta(t_* ,y) = \theta_* + \varepsilon.$$
This yields contradiction to the fact that $t_* = \inf \mathcal{T}$. Thus $\mathcal{T} = \emptyset$ and we have the desired conclusion.
\end{proof}

\vspace{0.2cm}

Now we are ready to provide a proof of Proposition \ref{P4.2}.  If we have $\mathcal{C}^1$-regularity to the global solution then we can get $L^\infty$-stability.
Suppose that there exists positive $\delta<\pi-\| \theta^0 - \tilde\theta^0 \|_\infty$ such that
\begin{gather*}
\mathcal{T}_\delta := \bigg\{t>0 : \|\theta(t,\cdot) - \tilde\theta(t,\cdot)\|_\infty +\int_0^t {\mathcal P}(s)ds \ge \|\theta^0 - \tilde\theta^0 \|_\infty + \delta \bigg\} \neq \emptyset.
\end{gather*}
We set 
\[ \tau :=\inf\mathcal{T}_\delta \in(0,\infty). \]
This implies
\begin{align} \label{D-2-1}
\|\theta(\tau,\cdot) - \tilde\theta(\tau,\cdot)\|_\infty +\int_0^{\tau} {\mathcal P}(s)ds = \|\theta^0 - \tilde\theta^0 \|_\infty + \delta  <\pi.
\end{align}
We fix $x\in D$ such that
$$ \| \theta(\tau,\cdot) - \tilde\theta(\tau,\cdot) \|_\infty = |\theta(\tau,x) - \tilde\theta(\tau,x)|.$$
It follows
\begin{align*}
\begin{aligned}
& \partial_t | \theta(\tau,x) - \tilde\theta(\tau,x) | \\
&\hspace{1cm} = \sgn(\theta(\tau,x) - \tilde\theta(\tau,x)) 
\bigg[  G\bigg( \nu(x) + \frac{\kappa}{\mu(D)} \int_D \phi(x,y) \sin(\theta(\tau,y) - \theta(\tau,x) )dy \bigg)
\\
&\hspace{3cm} -G\bigg(\nu(x) + \frac{\kappa}{\mu(D)}\int_D \phi(x,y) \sin(\tilde\theta(\tau,y) -\tilde\theta(\tau,x) ) \bigg) \bigg].
\end{aligned}
\end{align*}
Note that
\begin{align*}
\begin{aligned}
&\sin(\theta(\tau, y)-\theta(\tau, x)) - \sin(\tilde\theta(\tau, y)-\tilde\theta(\tau, x)) \\
& =2\cos \frac{\theta(\tau, y)-\theta(\tau, x) + (\tilde\theta(\tau, y)-\tilde\theta(\tau, x))}{2} \sin \frac{\theta(\tau, y)-\tilde\theta(\tau, y) - (\theta(\tau, x)-\tilde\theta(\tau, x))}{2}\\
& =2\cos\frac{\theta(\tau, y)-\theta(\tau, x) + (\tilde\theta(\tau, y)-\tilde\theta(\tau, x))}{2} \sin\frac{\theta(\tau, y) - \tilde\theta(\tau, y)}{2}\cos \frac{\theta(\tau, x) - \tilde\theta(\tau, x)}{2} \\
& \hspace{.2cm}-2\cos\frac{\theta(\tau, y)-\theta(\tau, x) + (\tilde\theta(\tau, y)-\tilde\theta(\tau, x))}{2}\cos \frac{\theta(\tau, y) - \tilde\theta(\tau, y)}{2} \sin\frac{\theta(\tau, x) - \tilde\theta(\tau, x)}{2}.
\end{aligned}
\end{align*}
It follows from Lemma \ref{L4.3} and \eqref{D-2-1} that
\begin{align*}
\begin{aligned}
&\sgn(\theta(\tau,x) - \tilde\theta(\tau,x)) \big[ \sin(\theta(\tau, y)-\theta(\tau, x)) - \sin(\tilde\theta(\tau, y)-\tilde\theta(\tau, x)) \big] \\
& \le 2\cos \frac{\theta(\tau, y)-\theta(\tau, x) + (\tilde\theta(\tau, y)-\tilde\theta(\tau, x))}{2}\sin\bigg|\frac{\theta(\tau, y) - \tilde\theta(\tau, y)}{2}\bigg|\cos\bigg|\frac{\theta(\tau, x) - \tilde\theta(\tau, x)}{2}\bigg|
\\
& \hspace{.2cm}-2\cos \frac{\theta(\tau, y)-\theta(\tau, x) + (\tilde\theta(\tau, y)-\tilde\theta(\tau, x))}{2}\cos\bigg|\frac{\theta(\tau, y) - \tilde\theta(\tau, y)}{2}\bigg|\sin\bigg|\frac{\theta(\tau, x) - \tilde\theta(\tau, x)}{2}\bigg|\\
& = 2\cos\frac{\theta(\tau, y)-\theta(\tau, x) + (\tilde\theta(\tau, y)-\tilde\theta(\tau, x))}{2} \sin \bigg( \bigg|\frac{\theta(\tau, y) - \tilde\theta(\tau, y)}{2}\bigg| -
\bigg|\frac{\theta(\tau, x) - \tilde\theta(\tau, x)}{2}\bigg| \bigg)\\
&\le 0.
\end{aligned}
\end{align*}
Thus, we have
\begin{align*}
\begin{aligned}
& \partial_t | \theta(\tau,x) - \tilde\theta(\tau,x) | \\
&\hspace{0.2cm} \leq \frac{m_{G'}\kappa}{\mu(D)} \int_D \phi(x,y) \sgn(\theta(\tau,x) - \tilde\theta(\tau,x)) \big[ \sin(\theta(\tau, y)-\theta(\tau, x)) - \sin(\tilde\theta(\tau, y)-\tilde\theta(\tau, x)) \big] dy \\
&\hspace{0.2cm} \le
\frac{2m_{G'} a_\phi \kappa\cos\theta_*}{\mu(D)} \int_D \sin \bigg( \bigg|\frac{\theta(\tau, y) - \tilde\theta(\tau, y)}{2}\bigg| - \bigg|\frac{\theta(\tau, x) - \tilde\theta(\tau, x)}{2}\bigg| \bigg) dy\\
&\hspace{0.2cm} \le \frac{m_{G'} a_\phi \kappa \sin\theta_*\cos\theta_*}{\theta_* \mu(D)} \int_D \big( \big|\theta(\tau, y) - \tilde\theta(\tau, y) \big| - \big|\theta(\tau, x) - \tilde\theta(\tau, x)\big| \big) dy\\
&\hspace{0.2cm} = \frac{m_{G'} a_\phi \kappa\sin2\theta_*}{2\theta_*\mu(D)} \big(\|\theta(\tau,\cdot) - \tilde\theta(\tau,\cdot)\|_{L^1} - \mu(D)\| \theta(\tau,\cdot) - \tilde\theta(\tau,\cdot)\|_{L^\infty} \big) \\
&\hspace{0.2cm} = -{\mathcal P}(\tau),
\end{aligned}
\end{align*}
where we used the following relations:
\begin{align*}
& \bigg| \bigg|\frac{\theta(\tau, y) - \tilde\theta(\tau, y)}{2}\bigg| - \bigg|\frac{\theta(\tau, x) - \tilde\theta(\tau, x)}{2}\bigg| \bigg| \leq \bigg| \frac{\theta(\tau, x) -\theta(\tau, y) +\tilde\theta(\tau, x)-\tilde\theta(\tau, y)}{2}\bigg| \leq \theta_*, \\
& \sin \bigg( \bigg|\frac{\theta(\tau, y) - \tilde\theta(\tau, y)}{2}\bigg| - \bigg|\frac{\theta(\tau, x) - \tilde\theta(\tau, x)}{2}\bigg| \bigg) \\
&\hspace{.5cm}\leq \frac{\sin\theta_*}{2\theta_*} \big( \big|\theta(\tau, y) - \tilde\theta(\tau, y) \big| - \big|\theta(\tau, x) - \tilde\theta(\tau, x)\big| \big).
\end{align*}
Hence, we obtain
\[
\frac{\partial}{\partial t}\bigg|_{t=\tau} \bigg( |\theta(t,x) - \tilde\theta(t,x)| + \int_0^t {\mathcal P}(s) ds \bigg) \le 0.
\]
Therefore, there exists $h\in(0,\tau)$ such that
\begin{align*}
\begin{aligned}
& \| \theta(\tau - h,\cdot) - \tilde\theta(\tau - h,\cdot) \|_{L^\infty} + \int_0^{\tau- h} {\mathcal P}(s) ds\\
& \hspace{0.5cm} \ge | \theta(\tau - h, x) - \tilde\theta(\tau -h, x)| +\int_0^{\tau- h} {\mathcal P}(s) ds \ge | \theta(\tau, x) - \tilde\theta(\tau, x)| +\int_0^{\tau}  {\mathcal P}(s)ds\\
& \hspace{0.5cm} = \|  \theta(\tau,\cdot) - \tilde\theta(\tau,\cdot) \|_{L^\infty} + \int_0^{\tau} {\mathcal P}(s)ds = \|\theta^0 - \tilde\theta^0 \|_{L^\infty} + \delta.
\end{aligned}
\end{align*}
This contradicts the fact that $\tau = \inf\mathcal{T}_\delta$. This completes the proof of Proposition \ref{P4.2}.

%%%%%%%%%%%%%%%%%%%%%%%%%%%%%
%
%
% Section 5
%
%
%%%%%%%%%%%%%%%%%%%%%%%%%%%%%
\section{The kinetic general Kuramoto model} \label{sec:5}
\setcounter{equation}{0} 
In this section, we present a uniform-in-time mean-field limit for the GK model \eqref{sys:GKM} to the kinetic GK model \eqref{sys:KGKM} in a measure theoretic framework.

\subsection{Measure theoretic framework}\label{sec:5.1} In this subsection, we provide a concept of global measure-valued solution to the kinetic GK model \eqref{sys:KGKM}, following the presentation in \cite{H-K-M-Rug-Z, H-K-P-Z, H-K-Z, H-Liu}. \newline

For $q\in[1, \infty)$, we define $(\Pi_q, d_q)$ to be a metric space where $\Pi_q$ is identical with $\bbt\times\bbr$ as a set and $d_q$ is defined by
\begin{align*}
d_q(x, \tilde x) = 
\begin{cases}
\displaystyle \Big( |\theta-\tilde\theta|^q+|\nu-\tilde\nu|^q \Big)^{\frac1q}, & q\in[1, \infty), \vspace{.1cm}\\
\displaystyle \max\big\{ |\theta-\tilde\theta|, |\nu-\tilde\nu| \big\}, & q=\infty,
\end{cases}
\quad x = (\theta, \nu), \quad \tilde x = (\tilde\theta, \tilde\nu),
\end{align*}
and let $\mathcal{P}(\Pi_q)$ be the set of all Radon probability measures on $\Pi_q$ and for $\rho\in\mathcal{P}(\Pi_q)$, we define
 \[
 \langle \rho,g \rangle:=\int_{\Pi_q} g(z)d\rho(z),\quad \forall ~g\in \mathcal C_0(\Pi_q),
 \]
where $\mathcal C_0$ denotes the set of continuous functions vanishing at infinity. Next, we recall the concept of measure-valued solution to the corresponding kinetic model. 
\begin{definition}\label{D5.1}
	For $T\in (0,\infty]$, $\rho_t\in L^\infty \big( [0,T);\mathcal{P}(\Pi_q) \big)$ is a measure-valued solution to the kinetic GK model \eqref{sys:KGKM} with the initial datum $\rho_0\in\mathcal{P}(\Pi_q)$, if and only if the following assertions hold:
\begin{enumerate}
\item $\{\rho_t \}_{t}$ is weakly continuous in $t$: 
\[
t\mapsto \langle\rho_t,g \rangle~\mbox{is continuous},\quad \forall ~g\in C_0(\Pi_q).
\]

\item $\rho$ satisfies \eqref{sys:KGKM} in a weak sense: for every $\varphi\in \mathcal C_0^1 \big( [0, T) \times \Pi_q \big)$,
\[
\langle \rho_t,\varphi(t,\cdot)\rangle-\langle \rho_0,\varphi(0,\cdot)\rangle=\int_{0}^{t}\left\langle\rho_s,\partial_s\varphi+{\mathfrak L}[\rho_s]\partial_\theta\varphi \right\rangle ds,
\]
where ${\mathfrak L}[\rho]$ is defined as 
\begin{align*}
{\mathfrak L}[\rho](\theta, \nu) := G\bigg( \nu +\kappa\int_{\Pi_q} \sin(\theta_*-\theta) d\rho(\theta_*, \nu_*) \bigg).
\end{align*}
\end{enumerate}
\end{definition}
\begin{remark} \label{R5.1}
Note that the probability distribution $f$ satisfying the kinetic GK model \eqref{sys:KGKM} in strong sense is also a measure-valued solution. Moreover, the empirical measure 
	\begin{equation}\label{particle}
	\rho_t^N:=\frac{1}{N}\sum_{\alpha=1}^N \delta_{\theta_\alpha(t)}\otimes \delta_{\nu_\alpha}
	\end{equation}
	becomes a measure-valued solution of the kinetic GK model \eqref{sys:KGKM}, provided that $\{\theta_\alpha\}_{\alpha=1}^N$ is a solution of the GK model \eqref{sys:GKM} with natural frequency $\{\nu_\alpha\}_{\alpha=1}^N$. Hence, one can rigorously derive the kinetic GK model \eqref{sys:KGKM} from the particle GK model \eqref{sys:GKM} by showing the convergence of a sequence of measure-valued solutions in suitable metric defined in the set of probability measures.
\end{remark}	

Now, we consider a subset of $\mathcal{P}(\Pi_q)$ with finite $q$-{th} moment and denote it by $\mathcal{P}_q(\Pi_q)$:
\[
\mathcal{P}_q(\Pi_q):=\bigg \{\rho: \int_{\Pi_q} d_q(z, z_0)^q ~d\rho(z) <\infty\quad \mbox{for some}~z_0\in \Pi_q \bigg \}.
\]
Then, the $q$-th Wasserstein distance $W_q(\cdot,\cdot)$ on $\mathcal{P}_q(\Pi_q)$ is defined as 
\begin{equation*}\label{Wasser}
\begin{aligned}
W_q (\rho_1, \rho_2) :=\inf_{\gamma\in \Gamma(\rho_1, \rho_2)}\bigg(\int_{\Pi_q\times\Pi_q} d_q(z_1, z_2)^q ~d\gamma(z_1, z_2) \bigg)^{\frac{1}{q}},
\end{aligned}
\end{equation*}
where $\Gamma(\rho_1,\rho_2)$ denotes the familiy of all joint probability measures on $\Pi_q\times\Pi_q$ with marginals $\rho_1$ and $\rho_2$.

\begin{definition} \label{D5.2}
\cite{Ne} The kinetic equation \eqref{sys:KGKM} is the mean-field limit of the particle system \eqref{sys:GKM}, which is valid in $[0,\tau)$, if for some $q\in [1,\infty)$ and $t\in [0,\tau)$, every measure-valued solution $\rho_t$ to \eqref{sys:KGKM} with the initial datum $\rho_0$ satisfies
\[
\lim_{N\to\infty}W_q(\rho_0^N,\rho_0)=0\quad \Longleftrightarrow \quad \lim_{N\to\infty} W_q(\rho_t^N,\rho_t)=0,
\] 
 where $\rho_t^N$ is a measure-valued solution given by \eqref{particle} with initial datum $\rho_0^N$. If $\tau = \infty$, we call the above mean-field limit as a uniform-in-time mean-field limit.
 \end{definition}
We close this subsection with introducing an approximation theorem of $q$-th Wasserstein distance without proofs.
\begin{proposition}\cite{Vi2} \label{P5.1}
 Wasserstein distance $W_q$ has the following properties:
 \begin{enumerate}
 \item $(\mathcal{P}_q(\Pi_q),W_q)$ is a metric space.
 \vspace{0.1cm}
 \item For any given $q\in [1,\infty)$ and $\rho\in \mathcal{P}_q(\Pi_q)$, there exists a sequence of empirical measures $\rho^n=\sum_{\alpha=1}^{n}a_{\alpha}\delta_{\theta=\theta_\alpha}\otimes\delta_{\nu=\nu_\alpha}\in \mathcal{P}_q(\Pi_q)$ such that 
\[
\lim_{N\to\infty}W_q (\rho,\rho^n)=0, \quad a_\alpha\in \mathbb{Q}, \quad \forall~\alpha=1,\cdots,n.
\] 
In particular, if $\rho$ has a compact support, $\rho^n$ can be chosen with additional constraint:
\[
\operatorname{supp}(\rho^n) \subset \operatorname{supp}(\rho), \quad \forall~n\geq1.
\]
\vspace{0.1cm}
\item Let $(\rho^n)_{n\geq1}$ be a sequence of probabilistic measures in $\mathcal{P}_q(\Pi_q)$ and $\rho$ be another element in $\mathcal{P}_q(\Pi_q)$. Then, $\rho^n$ converges to $\rho$ in $W_q$ if and only if
\begin{align*}
\lim_{n\to\infty} \int_{\Pi_q} \varphi(\theta, \nu) d\rho^n(\theta, \nu) = \int_{\Pi_q} \varphi(\theta, \nu) d\rho(\theta, \nu),
\end{align*}
for all continuous $\varphi$ satisfying
\begin{align} \label{E-1}
|\varphi(\theta, \nu)| \leq C(1+|\theta|^q+|\nu|^q), \quad \mbox{for some } C\in\bbr.
\end{align}
 \item For every complete subset $S$ of $\Pi_q$, the metric space  $(\mathcal{P}_q(S),W_q)$ is complete.
\end{enumerate}
\end{proposition}
\subsection{Measure-valued solution}\label{sec:5.2}
In this subsection, we present a rigorous justification of mean-field limit uniformly in time. Here, we refer the mean-field limit procedure in \cite{H-K-P-R-S1,H-K-P-Z, H-K-Z}.
\begin{theorem} \label{T5.1}
Suppose that there exist $\theta_*\in(0, \pi/2)$ and $\nu_l < \nu_r$ such that initial probability measure $\rho_0\in \mathcal{P}(\Pi_q)$ and coupling strengths $\kappa$ satisfy
\begin{align} \label{E-2}
\begin{aligned}
& \kappa \cos\theta_* \Big( \min_{F(a_G) \le y\le F(b_G)} G'(y) \Big)^2 \geq 2 \max_{(x,y)\in[a_G,b_G]^2} {\mathcal Q}^{-1}[F](x,y), \quad a_G = G(\nu_l-\kappa), \\
& b_G = G(\nu_r+\kappa), \quad \kappa \sin\theta_* > \nu_r-\nu_l, \quad \mathrm{supp}(\rho_0) \subset (0, \theta_*) \times (\nu_l, \nu_r) =: S,
\end{aligned}
\end{align}
where $ {\mathcal Q}^{-1}[F]($ is defined in \eqref{C-1}. Then, for $q\in[1, \infty)$, the following assertions hold. 
\begin{enumerate}
	\item There exists a unique measure-valued solution $\rho_t\in L^\infty([0,\infty);\mathcal{P}(\bar S))$ to \eqref{sys:KGKM} with the initial data $\rho_0$ such that $\rho_t$ can be approximated by a sequence of empirical measure-valued solution $\{\rho_t^N \}_{N\geq 1}$:
	\[\lim_{N\rightarrow\infty}\sup_{0 \leq t < \infty} W_{q}(\rho_t,\rho_t^N)=0. \]
	\item There exists a positive constant $\Lambda_4$ independent of $t$ such that for every two measure-valued solutions $\rho_t,\tilde\rho_t$ satisfying \eqref{E-2},  one has 
	\[W_{q}(\rho_t,\tilde\rho_t)\leq \Lambda_4 W_{q}(\rho_0,\tilde\rho_0),\quad \forall~t\geq 0. \]
\end{enumerate}
\end{theorem}
	\begin{proof} The proof basically follows four steps as in \cite{H-K-P-Z}. \newline
	
\noindent $\bullet$ Step A (Approximation of $\rho_0$ by empirical measures): As a consequence of Proposition \ref{P5.1}(2) and (3), we take a sequence $\{\tilde\rho_0^n\}_{n\geq1}$ of empirical measures such that 
\[
W_{q}(\rho_0, \tilde\rho_0^n) \longrightarrow 0, \quad \bar\nu^n := \int_{\Pi_q} \nu d\tilde\rho_0^n \longrightarrow \bar\nu := \int_{\Pi_q} \nu d\rho_0,
\]
as $n$ goes to infinity. Let us denote $\tilde\rho_0^n$ as
\begin{align*}
\tilde\rho_0^n = \frac{1}{N_n} \sum_{\alpha=1}^{N_n} \delta_{\theta^{n, 0}_\alpha} \otimes \delta_{\tilde\nu^n_\alpha}, \quad N_n\nearrow\infty,
\end{align*}
and define $\rho_0^n$ as
\begin{align} \label{E-3}
\rho_0^n := \frac{1}{N_n} \sum_{\alpha=1}^{N_n} \delta_{\theta^{n, 0}_\alpha} \otimes \delta_{\nu^n_\alpha}, \quad \nu^n_\alpha := \tilde\nu^n_\alpha+\bar\nu-\bar\nu^n.
\end{align}
Let $\varphi$ be a continuous function satisfying \eqref{E-1}. Then for any $\varepsilon>0$, there exists a positive integer $M_1 = M_1(\varepsilon, \varphi)$ such that
\begin{align*}
\bigg| \int_{\Pi_q} \varphi d\rho_0 -\frac{1}{N_n}\sum_{\alpha=1}^{N_n} \varphi(\theta^{n, 0}_\alpha, \tilde\nu^n_\alpha) \bigg| < \frac{\varepsilon}{2}, \quad \forall~n\geq M_1.
\end{align*}
In addition, there exists an $\varepsilon_0 > 0$ such that 
\begin{align*}
| \varphi(\theta_1, \nu_1) -\varphi(\theta_2, \nu_2)| < \frac{\varepsilon}{2}, 
\end{align*}
for all $x_i := (\theta_i, \nu_i)\in[0, \theta_*]\times[\nu_l, \nu_r]$, $i=1, 2$ satisfying $d_q(x_1, x_2) < \varepsilon_0$. Let $\varepsilon_1>0$ be a distance between $\mathrm{supp}(\rho_0)$ and the boundary of $(0, \theta_*)\times(\nu_l, \nu_r)$. We take a positive integer $M_2 = M_2(\varepsilon_0, \varepsilon_1)$ such that
\begin{align*}
|\bar\nu-\bar\nu^n| < \min\{ \varepsilon_0, \varepsilon_1 \}, \quad \forall~n\geq M_2.
\end{align*}
For all $n\geq\max\{M_1, M_2\}$, we obtain
\begin{align} \label{E-4}
\begin{aligned}
& \bigg| \int_{\bbt\times\bbr} \varphi d\rho_0 -\frac{1}{N_n}\sum_{\alpha=1}^{N_n} \varphi(\theta^{n, 0}_\alpha, \nu^n_\alpha) \bigg| \\
&\hspace{1cm} \leq \bigg| \int_{\bbt\times\bbr} \varphi d\rho_0 -\frac{1}{N_n}\sum_{\alpha=1}^{N_n} \varphi(\theta^{n, 0}_\alpha, \tilde\nu^n_\alpha) \bigg| +\frac{1}{N_n}\sum_{\alpha=1}^{N_n} \big| \varphi(\theta^{n, 0}_\alpha, \tilde\nu^n_\alpha) -\varphi(\theta^{n, 0}_\alpha, \nu^n_\alpha) \big| < \varepsilon.
\end{aligned}
\end{align}
As a result, we can take a sequence $\{\rho_0^n\}_{n\geq1}$ of empirical measures such that
\begin{align} \label{E-5}
\lim_{n\to\infty} W_{q}(\rho_0, \rho_0^n) = 0, \quad \int_{\Pi_q} \nu d\rho_0^n = \bar\nu, \quad \operatorname{supp}(\rho_0^n) \subset S, \quad \forall~n\geq1.
\end{align}
This follows from Proposition \ref{P5.1}, \eqref{E-3}, and \eqref{E-4}. Then for any $\varepsilon>0$, there exists a positive integer $N=N(\varepsilon)$ such that 
\[
W_{q}(\rho_0^m,\rho_0^n)<\varepsilon,\quad \mbox{for}~m,n>N(\varepsilon),
\]
since $(\mathcal{P}_q(\bar S),W_q)$ is a complete metric space. \\

\noindent $\bullet$ Step B (Approximation of $W_{q}(\rho_0^m,\rho_0^n)$): Problem of finding the value of $W_q^q(\rho_0^n,\rho_0^m)$ is equivalent to solving the optimization problem:
\begin{align*}
\mbox{minimize} \quad & \frac{1}{N_n N_m}\sum_{\alpha=1}^{N_n} \sum_{\beta=1}^{N_m} a_{\alpha\beta} \left( \big| \theta^{n, 0}_\alpha -\theta^{m, 0}_\beta \big|^q +\big| \nu^n_\alpha-\nu^m_\beta\big|^q \right), \\
\mbox{subjected to} \quad & \sum_{\alpha=1}^{N_n}a_{\alpha\beta}=N_n,\quad \sum_{\beta=1}^{N_m} a_{\alpha\beta}=N_m,\quad a_{\alpha\beta}\geq 0,
\end{align*}
and this linear programming problem always has a solution, namely, there exists $o_{\alpha\beta} \in \bbr$ such that
\begin{equation}\label{E-6}
W_q^q(\rho_0^n,\rho_0^m) = \frac{1}{N_nN_m}\sum_{\alpha=1}^{N_n} \sum_{\beta=1}^{N_m} o_{\alpha\beta} \left( \big| \theta^{n, 0}_\alpha -\theta^{m, 0}_\beta \big|^q +\big| \nu^n_\alpha-\nu^m_\beta\big|^q \right).
\end{equation}
Then, we approximate \eqref{E-6} with rational coefficients $r_{\alpha\beta}$ instead of real coefficients $o_{\alpha\beta}$, satisfying 
\begin{align*}
& |r_{\alpha\beta}-o_{\alpha\beta}|\leq \bigg( \frac{\varepsilon}{\pi/2+\nu_r-\nu_l} \bigg)^q, \quad r_{\alpha\beta}=\frac{M_{\alpha\beta}}{D_{mn}}, \\
& D_{mn}, M_{\alpha\beta}\in\mathbb{Z}_{\geq 0}, \quad \sum_{\alpha=1}^{N_n}r_{\alpha\beta}=N_n,\quad \sum_{\beta=1}^{N_m}r_{\alpha\beta}=N_m,
\end{align*}
so that $\{r_{\alpha\beta}\}_{\alpha,\beta}$ also represents a plan between $\rho_0^n$ and $\rho_0^m$.
Then, it follows from \eqref{E-6} that 
\begin{equation*}
\begin{aligned}
&\bigg|W_q^q(\rho_0^n,\rho_0^m)-\frac{1}{N_nN_m}\sum_{\alpha=1}^{N_n} \sum_{\beta=1}^{N_m} r_{\alpha\beta} \left( \big| \theta^{n, 0}_\alpha -\theta^{m, 0}_\beta \big|^q +\big| \nu^n_\alpha-\nu^m_\beta\big|^q \right) \bigg| \\
&\hspace{.5cm} \leq \frac{1}{N_n N_m} \sum_{\alpha=1}^{N_n} \sum_{\beta=1}^{N_m} |o_{\alpha\beta}-r_{\alpha\beta} | \left( \big| \theta^{n, 0}_\alpha -\theta^{m, 0}_\beta \big|^q +\big| \nu^n_\alpha-\nu^m_\beta\big|^q \right) \\
&\hspace{.5cm} \leq \varepsilon^q\cdot\frac{1}{N_nN_m}\sum_{\alpha=1}^{N_n} \sum_{\beta=1}^{N_m} \frac{\big| \theta^{n, 0}_\alpha -\theta^{m, 0}_\beta \big|^q +\big| \nu^n_\alpha-\nu^m_\beta\big|^q}{(\pi/2+\nu_r-\nu_l)^q} \leq \varepsilon^q,
\end{aligned}
\end{equation*}
where we used $\eqref{E-5}_3$ in the last inequality. Moreover, since each $r_{\alpha\beta}$ is rational number, we may rewrite the approximated value of $W_q^q(\rho_0^n,\rho_0^m)$ as 
\begin{equation}\label{E-7}
\begin{aligned}
& \frac{1}{N_nN_m}\sum_{\alpha=1}^{N_n} \sum_{\beta=1}^{N_m} r_{\alpha\beta} \left( \big| \theta^{n, 0}_\alpha -\theta^{m, 0}_\beta \big|^q +\big| \nu^n_\alpha-\nu^m_\beta\big|^q \right) \\
&\hspace{1cm} =\frac{1}{N_nN_mD_{nm}}\sum_{\alpha=1}^{N_n}\sum_{\beta=1}^{N_m} M_{\alpha\beta} \left( \big| \theta^{n, 0}_\alpha -\theta^{m, 0}_\beta \big|^q +\big| \nu^n_\alpha-\nu^m_\beta\big|^q \right).
\end{aligned}
\end{equation} 
Note that $\sum_{\alpha=1}^{N_n}\sum_{\beta=1}^{N_m} M_{\alpha\beta} = N_nN_mD_{nm} =: N_{nm}$. We define
\begin{align*}
& \theta^{n, m, 0}_{\alpha\beta\gamma} := \theta^{n, 0}_\alpha, \quad \nu^{n, m}_{\alpha\beta\gamma} := \nu^n_\alpha, \quad \tilde\theta^{n, m, 0}_{\alpha\beta\gamma} := \theta^{m, 0}_\beta, \quad \tilde\nu^{n, m}_{\alpha\beta\gamma} := \nu^m_\beta,
\end{align*}
for $\alpha=1, \cdots, N_n$, $\beta=1,\cdots,N_m$, and $\gamma=1,\cdots,M_{\alpha\beta}$. Then, one can rewrite \eqref{E-7} as
\begin{align*}
& \frac{1}{N_nN_m}\sum_{\alpha=1}^{N_n} \sum_{\beta=1}^{N_m} r_{\alpha\beta} \left( \big| \theta^{n, 0}_\alpha -\theta^{m, 0}_\beta \big|^q +\big| \nu^n_\alpha-\nu^m_\beta\big|^q \right) \\
&\hspace{1cm} = \frac{1}{N_{nm}} \sum_{\alpha=1}^{N_n} \sum_{\beta=1}^{N_m} \sum_{\gamma=1}^{M_{\alpha\beta}} \left( \big| \theta^{n, m, 0}_{\alpha\beta\gamma} -\tilde\theta^{n, m, 0}_{\alpha\beta\gamma} \big|^q +\big| \nu^{n,m}_{\alpha\beta\gamma} -\tilde\nu^{n, m}_{\alpha\beta\gamma} \big|^q \right).
\end{align*}
In summary, we have
\begin{align*}
\bigg|W_q^q(\rho_0^n,\rho_0^m) -\frac{1}{N_{nm}} \sum_{\alpha=1}^{N_n} \sum_{\beta=1}^{N_m} \sum_{\gamma=1}^{M_{\alpha\beta}} \left( \big| \theta^{n, m, 0}_{\alpha\beta\gamma} -\tilde\theta^{n, m, 0}_{\alpha\beta\gamma} \big|^q +\big| \nu^{n,m}_{\alpha\beta\gamma} -\tilde\nu^{n, m}_{\alpha\beta\gamma} \big|^q \right) \bigg| \leq \varepsilon^q.
\end{align*}

\noindent $\bullet$ Step C (Application of the $\ell_q$-stability estimate): 
Let $\{\theta^n_\alpha(t)\}_{\alpha=1}^{N_n}$ be a global smooth solution to the following Cauchy problem:
\begin{align*}
\begin{cases}
\displaystyle \dot\theta_\alpha = G\bigg( \nu^{n}_\alpha +\frac{\kappa}{N_n} \sum_{\beta=1}^{N_n} \sin(\theta_\beta - \theta_\alpha) \bigg), \quad \forall ~t>0, \vspace{.1cm} \\
\theta_\alpha \Big|_{t=0+} = \theta^{n, 0}_\alpha, \quad \forall ~\alpha = 1,\cdots,N_n,
\end{cases}
\end{align*}
and similarly, let $\{\theta^m_\alpha(t)\}_{\alpha=1}^{N_m}$ be a global smooth solution to the following Cauchy problem:
\begin{align*}
\begin{cases}
\displaystyle \dot\theta_\alpha = G\bigg( \nu^{m}_\alpha +\frac{\kappa}{N_m} \sum_{\beta=1}^{N_m} \sin(\theta_\beta - \theta_\alpha) \bigg), \quad \forall ~t>0, \vspace{.1cm} \\
\theta_\alpha \Big|_{t=0+} = \theta^{m, 0}_\alpha, \quad \forall ~\alpha = 1,\cdots,N_m.
\end{cases}
\end{align*}
Note that
\begin{align*}
\rho^n_t := \frac{1}{N_n} \sum_{\alpha=1}^{N_n} \delta_{\theta^n_\alpha(t)} \otimes \delta_{\nu^n_\alpha} \quad \mbox{and} \quad \rho^m_t := \frac{1}{N_m} \sum_{\alpha=1}^{N_m} \delta_{\theta^m_\alpha(t)} \otimes \delta_{\nu^m_\alpha}
\end{align*}
are measure-valued solutions of \eqref{sys:KGKM} with initial measures $\rho_0^n$ and $\rho_0^m$, respectively. If we define $\theta^{n, m}_{\alpha\beta\gamma}(t) := \theta^n_\alpha(t)$ and $\tilde\theta^{n, m}_{\alpha\beta\gamma}(t) := \theta^m_\beta(t)$ for $\alpha=1, \cdots, N_n$, $\beta=1,\cdots,N_m$, and $\gamma=1,\cdots,M_{\alpha\beta}$, then we can estimate $W_q^q(\rho_t^n,\rho_t^m)$ using the plane $\{r_{\alpha\beta}\}_{\alpha,\beta}$:
\begin{align*}
W_q^q(\rho_t^n,\rho_t^m) &\leq \frac{1}{N_nN_m} \sum_{\alpha=1}^{N_n} \sum_{\beta=1}^{N_m} r_{\alpha\beta} \left( \big| \theta^n_\alpha(t) -\theta^m_\beta(t) \big|^q +\big| \nu^n_\alpha -\nu^m_\beta \big|^q \right) \\
&= \frac{1}{N_{nm}} \sum_{\alpha=1}^{N_n} \sum_{\beta=1}^{N_m} \sum_{\gamma=1}^{M_{\alpha\beta}} \left( \big| \theta^{n, m}_{\alpha\beta\gamma}(t) -\tilde\theta^{n, m}_{\alpha\beta\gamma}(t) \big|^q +\big| \nu^{n, m}_{\alpha\beta\gamma} -\tilde\nu^{n, m}_{\alpha\beta\gamma} \big|^q \right).
\end{align*}
To apply the $\ell_q$-stability obtained in Section \ref{sec:3}, we note
\begin{align*}
\frac{d\theta^{n, m}_{\alpha\beta\gamma}}{dt} &= \frac{d\theta^n_\alpha}{dt} = G\bigg( \nu^{n}_\alpha +\frac{\kappa}{N_n} \sum_{\sigma_1=1}^{N_n} \sin\big( \theta^n_{\sigma_1}(t) - \theta^n_\alpha(t) \big) \bigg) \\
&= G\bigg( \nu^{n, m}_{\alpha\beta\gamma} +\frac{\kappa}{N_{nm}} \sum_{\sigma_1=1}^{N_n} N_mD_{nm} \sin\big( \theta^n_{\sigma_1}(t) - \theta^{n, m}_{\alpha\beta\gamma}(t) \big) \bigg) \\
&= G\bigg( \nu^{n, m}_{\alpha\beta\gamma} +\frac{\kappa}{N_{nm}} \sum_{\sigma_1=1}^{N_n} \sum_{\sigma_2=1}^{N_m} \sum_{\sigma_3=1}^{M_{\sigma_1\sigma_2}} \sin\big( \theta^{n, m}_{\sigma_1\sigma_2\sigma_3}(t) - \theta^{n, m}_{\alpha\beta\gamma}(t) \big) \bigg)
\end{align*}
to find that $\big\{ \theta^{n, m}_{\alpha\beta\gamma} \big\}_{\alpha,\beta,\gamma}$ is a solution of the following Cauchy problem:
\begin{align*}
\begin{cases}
\displaystyle \dot\theta_{\alpha\beta\gamma} = G\bigg( \nu^{n, m}_{\alpha\beta\gamma} +\frac{\kappa}{N_{nm}} \sum_{\sigma_1=1}^{N_n} \sum_{\sigma_2=1}^{N_m} \sum_{\sigma_3=1}^{M_{\sigma_1\sigma_2}} \sin\big( \theta_{\sigma_1\sigma_2\sigma_3} - \theta_{\alpha\beta\gamma} \big) \bigg), \quad \forall ~t>0, \vspace{.1cm} \\
\theta_{\alpha\beta\gamma} \Big|_{t=0+} = \theta^{n, m, 0}_{\alpha\beta\gamma}, \quad \forall ~\alpha = 1,\cdots,N_n, \quad \forall~\beta=1,\cdots,N_m, \quad \forall~\gamma=1,\cdots,M_{\alpha\beta}.
\end{cases}
\end{align*}
In the same way, $\big\{ \tilde\theta^{n, m}_{\alpha\beta\gamma} \big\}_{\alpha,\beta,\gamma}$ becomes a solution of the following Cauchy problem:
\begin{align*}
\begin{cases}
\displaystyle \dot\theta_{\alpha\beta\gamma} = G\bigg( \tilde\nu^{n, m}_{\alpha\beta\gamma} +\frac{\kappa}{N_{nm}} \sum_{\sigma_1=1}^{N_n} \sum_{\sigma_2=1}^{N_m} \sum_{\sigma_3=1}^{M_{\sigma_1\sigma_2}} \sin\big( \theta_{\sigma_1\sigma_2\sigma_3} - \theta_{\alpha\beta\gamma} \big) \bigg), \quad \forall ~t>0, \vspace{.1cm} \\
\theta_{\alpha\beta\gamma} \Big|_{t=0+} = \tilde\theta^{n, m, 0}_{\alpha\beta\gamma}, \quad \forall ~\alpha = 1,\cdots,N_n, \quad \forall~\beta=1,\cdots,N_m, \quad \forall~\gamma=1,\cdots,M_{\alpha\beta}.
\end{cases}
\end{align*}
Hence, one can use Corollary \ref{C3.1} to obtain
\begin{align*}
& W_q^q(\rho_t^n,\rho_t^m) \leq \frac{1}{N_{nm}} \sum_{\alpha=1}^{N_n} \sum_{\beta=1}^{N_m} \sum_{\gamma=1}^{M_{\alpha\beta}} \big| \theta^{n, m}_{\alpha\beta\gamma}(t) -\tilde\theta^{n, m}_{\alpha\beta\gamma}(t) \big|^q +\frac{1}{N_{nm}} \sum_{\alpha=1}^{N_n} \sum_{\beta=1}^{N_m} \sum_{\gamma=1}^{M_{\alpha\beta}} \big| \nu^{n, m}_{\alpha\beta\gamma} -\tilde\nu^{n, m}_{\alpha\beta\gamma} \big|^q \\
&\hspace{.5cm}\leq \frac{\Lambda_3^q}{N_{nm}} \left[ \bigg( \sum_{\alpha=1}^{N_n} \sum_{\beta=1}^{N_m} \sum_{\gamma=1}^{M_{\alpha\beta}} \big| \theta^{n, m, 0}_{\alpha\beta\gamma} -\tilde\theta^{n, m, 0}_{\alpha\beta\gamma} \big|^q \bigg)^{\frac1q} +\bigg( \sum_{\alpha=1}^{N_n} \sum_{\beta=1}^{N_m} \sum_{\gamma=1}^{M_{\alpha\beta}} \big| \nu^{n, m}_{\alpha\beta\gamma} -\tilde\nu^{n, m}_{\alpha\beta\gamma} \big|^q \bigg)^{\frac1q} \right]^q\\
&\hspace{.7cm} +\frac{1}{N_{nm}} \sum_{\alpha=1}^{N_n} \sum_{\beta=1}^{N_m} \sum_{\gamma=1}^{M_{\alpha\beta}} \big| \nu^{n, m}_{\alpha\beta\gamma} -\tilde\nu^{n, m}_{\alpha\beta\gamma} \big|^q \\
&\hspace{.7cm}\leq \frac{1+2^{q-1} \Lambda_3^q}{N_{nm}} \sum_{\alpha=1}^{N_n} \sum_{\beta=1}^{N_m} \sum_{\gamma=1}^{M_{\alpha\beta}} \Big( \big| \theta^{n, m, 0}_{\alpha\beta\gamma} -\tilde\theta^{n, m, 0}_{\alpha\beta\gamma} \big|^q +\big| \nu^{n, m}_{\alpha\beta\gamma} -\tilde\nu^{n, m}_{\alpha\beta\gamma} \big|^q \Big) \\
&\hspace{.5cm} \leq \big( 1+2^{q-1} \Lambda_3^q \big) \cdot \big(W_q^q(\rho_0^n,\rho_0^m) +\varepsilon^q \big),
\end{align*}
where we used the result from the previous step in the last inequality. Therefore, for $n, m > N(\varepsilon)$, we have
\begin{align*}
W_q(\rho_t^n,\rho_t^m) \leq \big( 2+2^q \Lambda_3^q \big)^{\frac1q} \varepsilon \leq 2(1+\Lambda_3) \varepsilon,
\end{align*}
which implies that $\{\rho_t^n \}_{n\geq 1}$ is Cauchy in $\big( \mathcal{P}_q(\bar S), W_{q} \big)$. Since this metric space is complete, we can find a limit measure $\rho_t$: 
\begin{align*}
W_q(\rho_t^n,\rho_t)\leq 2(1+\Lambda_3) \varepsilon, \quad \forall~n>N(\varepsilon) \implies \lim_{n\to\infty}\sup_{t\geq0}W_{q}(\rho_t^n,\rho_t)=0.
\end{align*}
Since $W_q$-convergence implies the weak-convergence, one-parameter family $\{\rho_t\}_{t\geq 0}$ is indeed a measure-valued solution of \eqref{sys:KGKM}, by using Lebesgue dominant convergence theorem. Since it is easy to see the weak continuity in Definition \ref{D5.1}, we only show the second condition in Definition \ref{D5.1}. It is enough to show that
for $\varphi\in \mathcal{C}^1_0([0,T)\times \Pi_q)$,
\[
\lim_{n\to \infty}\int_0^t \langle \rho_s^n , \partial_s\varphi + {\mathfrak L}[\rho_s^n] \partial_\theta \varphi \rangle =
\int_0^t \langle \rho_s , \partial_s\varphi + {\mathfrak L}[\rho_s] \partial_\theta \varphi \rangle.
\]
Note that 
\begin{align*}
&\lim_{n\to \infty}\int_0^t \langle \rho_s^n , \partial_s\varphi + {\mathfrak L}[\rho_s^n] \partial_\theta \varphi \rangle 
= \lim_{n\to \infty}  \int_0^t  \int_{\Pi_q} \big( \partial_s\varphi + {\mathfrak L}[\rho_s^n] \partial_\theta \varphi \big) d\rho^n_s ds \\
&\hspace{0.5cm} =   \int_0^t \lim_{n\to \infty} \left[ \int_{\Pi_q} \big( \partial_s\varphi + {\mathfrak L}[\rho_s^n] \partial_\theta \varphi \big) d\rho^n_s \right]  ds 
{=}   \int_0^t \left( \partial_s\varphi + \lim_{n\to \infty}  \int_{\Pi_q} \big(  {\mathfrak L}[\rho_s^n] \partial_\theta \varphi \big) d\rho^n_s  \right)ds,
\end{align*}
where we used the Lebesgue dominated convergence theorem in the last second equality. Hence, all we need to show is
\[
	 \lim_{n\to\infty} \int_{\Pi_q} \big(  {\mathfrak L}[\rho_s^n] \partial_\theta \varphi \big) d\rho^n_s
	 = \int_{\Pi_q} \big(  {\mathfrak L}[\rho_s] \partial_\theta \varphi \big) d\rho_s.
\]
We observe
\begin{align*}
& \int_{\Pi_q} \big( {\mathfrak L}[\rho_s^n] \partial_\theta \varphi \big) d\rho^n_s 
-  \int_{\Pi_q} \big(  {\mathfrak L}[\rho_s] \partial_\theta \varphi \big) d\rho_s\\
&\hspace{1cm} = \underbrace{\int_{\Pi_q} \big(  {\mathfrak L}[\rho_s^n] - {\mathfrak L}[\rho_s] \big) \partial_\theta \varphi d\rho^n_s  }_{=:\mathcal{J}_1}
+ \underbrace{ \int_{\Pi_q} \big(  {\mathfrak L}[\rho_s] \partial_\theta \varphi \big) d\rho^n_s
-  \int_{\Pi_q} \big(  {\mathfrak L}[\rho_s] \partial_\theta \varphi \big) d\rho_s}_{=:\mathcal{J}_2}.
\end{align*}
The term $\mathcal{J}_2$ goes to zero as $n$ goes to the infinity. For the term $\mathcal{J}_1$, we have
\begin{align*}
\left| \int_{\Pi_q} \big(  {\mathfrak L}[\rho_s^n] - {\mathfrak L}[\rho_s] \big) \partial_\theta \varphi d\rho^n_s\right|
\leq \sup_{(\theta,\nu)} |({\mathfrak L}[\rho_s^n](\theta,\nu) - {\mathfrak L}[\rho_s](\theta,\nu)) \partial_\theta\varphi(\theta,\nu)|,
\to 0
\end{align*}
as $n$ goes to infinity, since
\begin{align*}
&\sup_{(\theta,\nu)} |({\mathfrak L}[\rho_s^n](\theta,\nu) - {\mathfrak L}[\rho_s](\theta,\nu)) \partial_\theta\varphi(\theta,\nu)|\\
& \hspace{0.5cm} = \sup_{(\theta,\nu)} |\partial_\theta\varphi(\theta,\nu)|\left| G\left( \nu + \kappa \int_{\Pi_q} \sin(\theta_* -\theta) d\rho^n_s(\theta_*,\nu_*)\right) \right.\\
&\hspace{3.5cm}\left.-G\left( \nu + \kappa \int_{\Pi_q} \sin(\theta_* -\theta) d\rho_s(\theta_*,\nu_*)\right) 
\right| \\
&\hspace{0.5cm} \leq M_{\partial_\theta \varphi} M_{G'} \kappa \sup_{(\theta,\nu)} \left| \int_{\Pi_q} \sin(\theta_* -\theta) d\rho_s^n -\int_{\Pi_q} \sin(\theta_* -\theta) d\rho_s \right| \\
&\hspace{0.5cm} \leq M_{\partial_\theta \varphi} M_{G'} \kappa \sup_{(\theta,\nu)} \left|
\cos\theta \left( \int_{\Pi_q} \sin\theta_* d\rho_s^n - \int_{\Pi_q} \sin\theta_* d\rho_s  \right)\right.\\
&\hspace{3.5cm}\left.-\sin\theta \left( \int_{\Pi_q} \cos\theta_* d\rho_s^n - \int_{\Pi_q} \cos\theta_* d\rho_s   \right)
\right| \\
&\hspace{0.5cm} \lesssim \left| \int_{\Pi_q} \sin\theta_* d\rho_s^n - \int_{\Pi_q} \sin\theta_* d\rho_s  \right|
+\left| \int_{\Pi_q} \cos\theta_* d\rho_s^n - \int_{\Pi_q} \cos\theta_* d\rho_s   \right|,
\end{align*}
where $\sup_{(\theta,\nu)} |\partial_\theta\varphi(\theta,\nu)| =: M_{\partial_\theta \varphi}$.

\vspace{0.2cm}

\noindent $\bullet$ Step D (Uniform-in-time stability of kinetic equation): For given two measures $\rho_0$ and $\tilde\rho_0$ in $\mathcal{P}(\bar S)$, let $\rho_t$ and $\tilde\rho_t$ be two measure-valued solutions of \eqref{sys:KGKM} subject to initial data $\rho_0$ and $\tilde\rho_0$, respectively. Then, for any  positive $\varepsilon\ll 1$, find positive integer $\tilde N = \tilde N(\varepsilon)$ such that 
\[W_{q}(\rho,\rho^n_t)<\frac{\varepsilon}{2},\quad W_{q}(\tilde\rho_t,\tilde\rho_t^n)<\frac{\varepsilon}{2},\quad n\geq \tilde N(\varepsilon). \]
Moreover, one can use same argument in the previous step to obtain
\begin{align*}
W_{q}^q(\rho_t^n,\tilde\rho_t^n) \leq \big( 1+2^{q-1} \Lambda_3^q \big) \cdot \big(W_q^q(\rho_0^n,\tilde\rho_0^n) +\varepsilon^q \big).
\end{align*}
Hence, we have
\[\begin{aligned}
W_{q}^q(\rho_t,\tilde\rho_t)&\leq \Big( W_{q}(\rho_t,\rho_t^n)+W_{q}(\rho_t^n,\tilde\rho_t^n)+W_{q}(\tilde\rho_t^n, \tilde\rho_t) \Big)^q\\
&\leq \Big( \varepsilon+W_{q}(\rho_t^n, \tilde\rho_t^n) \Big)^q \leq 2^{q-1} \Big(\varepsilon^q+W_{q}^q(\rho_t^n,\tilde\rho_t^n) \Big)\\
&\leq 2^{q-1}\varepsilon^q +2^{q-1} \big( 1+2^{q-1} \Lambda_3^q \big) \cdot \big(W_q^q(\rho_0^n,\tilde\rho_0^n) +\varepsilon^q \big)
\end{aligned} \]
We let $\varepsilon$ go to $0$, and then let $n$ go to infinity to obtain the uniform $W_{q}$-stability: 
\[W_q(\rho_t,\tilde\rho_t)\leq \big( 2^{\frac{q-1}{q}}+4^{\frac{q-1}{q}} \Lambda_3 \big) W_q(\rho_0, \tilde\rho_0), \quad \forall~t\geq 0. \]
This also implies the uniqueness of the measure-valued solution. 
\end{proof}

%%%%%%%%%%%%%%%%%%%%%%%%%%%%%
%
%
% Section 6
%
%
%%%%%%%%%%%%%%%%%%%%%%%%%%%%%
\section{Conclusion} \label{sec:6}
\setcounter{equation}{0}
In this paper, we have  presented three issues for the GK model. First, we provided the uniform $\ell_p$-stability of the particle GK model. Although there are no conservation laws, we use the exponentially fast convergence of velocity to bound the $\ell_p$-distance between two phase configurations with different initial data, natural frequency, and communication matrix. Second, we formally derived the continuum GK model from the lattice GK model and proved the uniform-in-time continuum limit from the lattice GK model to the derived continuum GK model. Moreover, we showed that the continuum model exhibits $L^{\infty}$-contraction which coincides with analogous results for the lattice model.  Third, we have derived the kinetic GK model from the GK model and proved the uniform-in-time mean-field limit using the uniform stability estimate. All frameworks presented in this paper are sufficient ones, so certainly our results are not optimal. There will be some rooms for improvements. For example, uniform stability estimates hold for some class of restricted initial data, and  the derivation of uniform-in-time asymptotic limits depends on the uniform stability estimate, hence they are valid only for restricted initial data. For generic initial data, it is not clear whether uniform-in-time asymptotic limits hold or not. Some instability issues might also emerge. We leave this issue for a future work.

%%%%%%%%%%%%%%%%%%%%%%%%%%%%%
%
%
% Appendix
%
%
%%%%%%%%%%%%%%%%%%%%%%%%%%%%%
\appendix
\section{Estimate (\textit{ii}) in Lemma \ref{L3.2}} \label{App-A}
\setcounter{equation}{0}
First, we split $\cI_{13}$ into two terms:
\begin{align}
\begin{aligned} \label{Ap-1}
\mathcal{I}_{13} &= \frac{p\kappa}{N}\sum_{\alpha,\beta=1}^N \phi_{\alpha\beta} \sgn\big( \omega_\alpha - \tilde\omega_\alpha \big) \big| \omega_\alpha - \tilde\omega_\alpha \big|^{p-1} \cos(\tilde\theta_\beta - \tilde\theta_\alpha) G'(\nu_c) \big((\omega_\beta - \omega_\alpha) - (\tilde\omega_\beta - \tilde\omega_\alpha) \big) \\
&\hspace{.2cm} +\frac{p\kappa}{N} \sum_{\alpha,\beta=1}^N \phi_{\alpha\beta} \sgn\big( \omega_\alpha - \tilde\omega_\alpha \big) \big| \omega_\alpha - \tilde\omega_\alpha \big|^{p-1} \cos(\tilde\theta_\beta - \tilde\theta_\alpha) \bigg( \frac{1}{F'(\tilde\omega_\alpha)} - G'(\nu_c)\bigg) \\ &\hspace{.4cm}\times\big((\omega_\beta - \omega_\alpha) - (\tilde\omega_\beta - \tilde\omega_\alpha) \big) \\
&=: \mathcal{I}_{131} + \mathcal{I}_{132}.
\end{aligned}
\end{align}
Next, we estimate the term ${\mathcal I}_{13i}$ one by one.  \newline

\noindent $\bullet$~(Estimate of $\mathcal{I}_{131}$): Note that 
\[
 \big((\omega_\alpha - \tilde\omega_\alpha) - (\omega_\beta - \tilde\omega_\beta) \big) \big( \sgn\big( \omega_\alpha - \tilde\omega_\alpha \big) \big| \omega_\alpha - \tilde\omega_\alpha \big|^{p-1} - \sgn\big( \omega_\beta - \tilde\omega_\beta \big) \big| \omega_\beta - \tilde\omega_\beta \big|^{p-1} \big) \ge 0,
\]
and we set 
\[ C_{131} = G'(\nu_c)m_\Phi\cos\theta_*. \] 
Then, one has 
\begin{align}
\begin{aligned} \label{Ap-2}
&\mathcal{I}_{131} = -\frac{p\kappa G'(\nu_c)}{2N}\sum_{\alpha,\beta=1}^N \Big[ \phi_{\alpha\beta}\cos(\tilde\theta_\beta - \tilde\theta_\alpha)  \big((\omega_\alpha - \tilde\omega_\alpha) - (\omega_\beta - \tilde\omega_\beta)\big)\\
&\hspace{2cm}\times \Big( \sgn\big( \omega_\alpha - \tilde\omega_\alpha \big) \big| \omega_\alpha - \tilde\omega_\alpha \big|^{p-1} - \sgn\big( \omega_\beta - \tilde\omega_\beta \big) \big| \omega_\beta - \tilde\omega_\beta \big|^{p-1} \Big)  \Big] \\
&\le -\frac{p\kappa C_{131}}{2N}\sum_{\alpha,\beta=1}^N \Big[  \Big(  {\mathcal Q}^{-1}[F](\omega_\alpha, \tilde\omega_\alpha) \big(F(\omega_\alpha) -F(\tilde\omega_\alpha)\big) - {\mathcal Q}^{-1}[F](\omega_\beta, \tilde\omega_\beta) \big( F(\omega_\beta) -F(\tilde\omega_\beta) \big) \Big) \\
&\hspace{2cm}\times\Big( \sgn\big( \omega_\alpha - \tilde\omega_\alpha \big) \big| \omega_\alpha - \tilde\omega_\alpha \big|^{p-1} - \sgn\big( \omega_\beta - \tilde\omega_\beta \big) \big| \omega_\beta - \tilde\omega_\beta \big|^{p-1} \Big)  \Big]\\
&= -\frac{p\kappa G'(\nu_c) C_{131}}{2N}\sum_{\alpha,\beta=1}^N \Big[ \Big( \sgn\big( \omega_\alpha - \tilde\omega_\alpha \big) \big| \omega_\alpha - \tilde\omega_\alpha \big|^{p-1} - \sgn\big( \omega_\beta - \tilde\omega_\beta \big) \big| \omega_\beta - \tilde\omega_\beta \big|^{p-1} \Big) \\
&\hspace{4cm}\times  \Big( \big[F(\omega_\alpha) -F(\tilde\omega_\alpha)\big] -\big[ F(\omega_\beta) -F(\tilde\omega_\beta) \big] \Big) \Big] \\
&\hspace{.2cm} -\frac{p\kappa C_{131}}{2N} \sum_{\alpha,\beta=1}^N \Big[ \Big( \sgn\big( \omega_\alpha - \tilde\omega_\alpha \big) \big| \omega_\alpha - \tilde\omega_\alpha \big|^{p-1} - \sgn\big( \omega_\beta - \tilde\omega_\beta \big) \big| \omega_\beta - \tilde\omega_\beta \big|^{p-1} \Big) \\
&\hspace{3cm} \times \Big( \big[  {\mathcal Q}^{-1}[F](\omega_\alpha, \tilde\omega_\alpha) -G'(\nu_c) \big] \big[F(\omega_\alpha) -F(\tilde\omega_\alpha)\big] \\
&\hspace{3.7cm} -\big[  {\mathcal Q}^{-1}[F](\omega_\beta, \tilde\omega_\beta) -G'(\nu_c) \big] \big[ F(\omega_\beta) -F(\tilde\omega_\beta) \big] \Big) \Big] \\
&=: \cI_{1311} +\cI_{1312}.
\end{aligned}
\end{align}
Below, we further estimate the term ${\mathcal I}_{131i}$. \newline

\noindent $\diamond$~(Estimate of $\cI_{1311}$):~We use \eqref{C-3-2} and
\begin{align*}
\sum_{\alpha=1}^N \big( F(\omega_\alpha) -F(\tilde\omega_\alpha) \big) = \sum_{\alpha=1}^N (\nu_\alpha-\tilde\nu_\alpha) = 0
\end{align*}
to see
\begin{align}
\begin{aligned} \label{Ap-3}
\cI_{1311} &= -\frac{p\kappa G'(\nu_c)C_{131}}{2N}\sum_{\alpha,\beta=1}^N \Big[ \Big( \sgn\big( \omega_\alpha - \tilde\omega_\alpha \big) \big| \omega_\alpha - \tilde\omega_\alpha \big|^{p-1} - \sgn\big( \omega_\beta - \tilde\omega_\beta \big) \big| \omega_\beta - \tilde\omega_\beta \big|^{p-1} \Big) \\
&\hspace{4cm}\times  \Big( \big[F(\omega_\alpha) -F(\tilde\omega_\alpha)\big] -\big[F(\omega_\beta) -F(\tilde\omega_\beta) \big] \Big) \Big] \\
&= -p\kappa G'(\nu_c) C_{131} \sum_{\alpha=1}^N \big|F(\omega_\alpha) -F(\tilde\omega_\alpha)\big| \cdot \big| \omega_\alpha - \tilde\omega_\alpha \big|^{p-1} \\
&\leq -\frac{p\kappa G'(\nu_c) C_{131}}{M_{ {\mathcal Q}^{-1}[F]}} \|\Omega - \tilde\Omega\|^p_p.
\end{aligned}
\end{align}
\vspace{0.2cm}

\noindent $\diamond$~(Estimate of $\cI_{1312}$):~we choose  $\tilde c_\alpha$ between $\omega_\alpha$ and $\tilde\omega_\alpha$ such that
\begin{align} 
\begin{aligned} \label{Ap-4}
&\big|  {\mathcal Q}^{-1}[F] \big( \omega_\alpha, \tilde\omega_\alpha \big) -G'(\nu_c) \big| \\
& \hspace{0.5cm} = \big| G'\big( F(\tilde c_\alpha) \big) -G'(\nu_c) \big| = \big| {\mathcal Q}[G^{\prime}] \big( F(\tilde c_\alpha), \nu_c \big) \big( F(\tilde c_\alpha) -\nu_c \big) \big| \\
&\hspace{.5cm} \leq M_{{\mathcal Q}[G^{\prime}] } \big| F(\tilde c_\alpha) -\nu_c \big| \leq M_{{\mathcal Q}[G^{\prime}]} \big( |F(\omega_\alpha) -\nu_c| +|F(\tilde\omega_\alpha) -\nu_c| \big) \\
& \hspace{0.5cm} \leq \frac{M_{{\mathcal Q}[G^{\prime}] }}{m_{ {\mathcal Q}^{-1}[F]}} \big[ \mathcal{D}(\Omega^0) +\mathcal{D}(\tilde\Omega^0) \big] e^{-\Lambda_1 t},
\end{aligned}
\end{align} 
where we used Lemma \ref{L3.1} in the last inequality. Hence, $\cI_{1312}$ can be estimated by
\begin{align}
\begin{aligned} \label{Ap-4-1}
\cI_{1312} &= -\frac{p\kappa C_{131}}{2N}\sum_{\alpha,\beta=1}^N \Big[ \Big( \sgn\big( \omega_\alpha - \tilde\omega_\alpha \big) \big| \omega_\alpha - \tilde\omega_\alpha \big|^{p-1} - \sgn\big( \omega_\beta - \tilde\omega_\beta \big) \big| \omega_\beta - \tilde\omega_\beta \big|^{p-1} \Big) \\
&\hspace{3.2cm}\times \Big( \big[ {\mathcal Q}^{-1}[F] (\omega_\alpha, \tilde\omega_\alpha) -G'(\nu_c) \big] \big[F(\omega_\alpha) -F(\tilde\omega_\alpha)\big] \\
&\hspace{4cm} -\big[ {\mathcal Q}^{-1}[F](\omega_\beta, \tilde\omega_\beta) -G'(\nu_c) \big] \big[ F(\omega_\beta) -F(\tilde\omega_\beta) \big] \Big) \Big] \\
&\leq \frac{p\kappa C_{131}}{N}\sum_{\alpha,\beta=1}^N \big| \omega_\alpha - \tilde\omega_\alpha \big|^{p-1} \big| {\mathcal Q}^{-1}[F](\omega_\alpha, \tilde\omega_\alpha) -G'(\nu_c) \big| \big|F(\omega_\alpha) -F(\tilde\omega_\alpha)\big| \\
&\hspace{.2cm} +\frac{p\kappa C_{131}}{N}\sum_{\alpha,\beta=1}^N \big| \omega_\alpha - \tilde\omega_\alpha \big|^{p-1} \big| {\mathcal Q}^{-1}[F](\omega_\beta, \tilde\omega_\beta) -G'(\nu_c) \big| \big|F(\omega_\beta) -F(\tilde\omega_\beta)\big| \\
&\overset{\eqref{Ap-4}}{\leq} \frac{p\kappa M_{{\mathcal Q}[G^{\prime}]} C_{131}}{Nm_{{\mathcal Q}^{-1}[F]}} \big[ \mathcal{D}(\Omega^0) +\mathcal{D}(\tilde\Omega^0) \big] e^{-\Lambda_1 t} \sum_{\alpha,\beta=1}^N \big| \omega_\alpha - \tilde\omega_\alpha \big|^{p-1} \big|F(\omega_\alpha) -F(\tilde\omega_\alpha)\big| \\
&\hspace{.2cm} +\frac{p\kappa M_{{\mathcal Q}[G^{\prime}]} C_{131}}{Nm_{{\mathcal Q}^{-1}[F]}} \big[ \mathcal{D}(\Omega^0) +\mathcal{D}(\tilde\Omega^0) \big] e^{-\Lambda_1 t} \sum_{\alpha,\beta=1}^N \big| \omega_\alpha - \tilde\omega_\alpha \big|^{p-1} \big|F(\omega_\beta) -F(\tilde\omega_\beta)\big| \\
&\leq \frac{12p\kappa^2 b_\Phi^2 M_{G'}^2 M_{{\mathcal Q}[G^{\prime}]}}{m^2_{{\mathcal Q}^{-1}[F]}} e^{-\Lambda_1 t} \|\Omega - \tilde\Omega\|^p_p,
\end{aligned}
\end{align}
where the last inequality can be obtained by \eqref{C-6} and
\begin{align}
\begin{aligned}  \label{Ap-5}
& \sum_{\alpha,\beta=1}^N \big| \omega_\alpha - \tilde\omega_\alpha \big|^{p-1} \big| \omega_\beta-\tilde\omega_\beta \big| \\
&\hspace{0.5cm} \leq \bigg( \sum_{\alpha,\beta=1}^N |\omega_\alpha - \tilde\omega_\alpha |^{p-1\frac{p}{p-1}}  \bigg)^{\frac{p-1}{p}} \bigg(\sum_{\alpha,\beta=1}^N |\omega_\beta - \tilde\omega_\beta|^p \bigg)^\frac{1}{p} = N \|\Omega - \tilde\Omega\|^p_p.
\end{aligned}
\end{align}
In \eqref{Ap-2}, we combine \eqref{Ap-3} and  \eqref{Ap-4-1} to find 
\begin{equation} \label{Ap-5-1}
\mathcal{I}_{131} \leq -\frac{p\kappa G'(\nu_c) C_{131}}{M_{ {\mathcal Q}^{-1}[F]}} \|\Omega - \tilde\Omega\|^p_p + \frac{12p\kappa^2 b_\Phi^2 M_{G'}^2 M_{{\mathcal Q}[G^{\prime}]}}{m^2_{{\mathcal Q}^{-1}[F]}} e^{-\Lambda_1 t} \|\Omega - \tilde\Omega\|^p_p.
\end{equation}

\vspace{0.5cm}

\noindent $\bullet$~(Estimate of $\mathcal{I}_{132}$):~For this, we use Lemma \ref{L3.1} to find
\[
\bigg| \frac{1}{F'(\tilde\omega_\alpha)} - G'(\nu_c)\bigg| = \big| G'(F(\tilde\omega_\alpha)) - G'(\nu_c) \big| \le M_{{\mathcal Q}[G^{\prime}]}\big| F(\tilde\omega_\alpha) -\nu_c \big| \le \frac{M_{{\mathcal Q}[G^{\prime}]}}{m_{{\mathcal Q}^{-1}[F]}} \mathcal{D}(\tilde\Omega^0) e^{-\Lambda_1 t}.  
\]
This yields
\begin{align}
\begin{aligned} \label{Ap-6}
|\mathcal{I}_{132}| &\leq \frac{p\kappa}{N} \bigg| \sum_{\alpha,\beta=1}^N \phi_{\alpha\beta} \big| \omega_\alpha - \tilde\omega_\alpha \big|^{p-1}  \bigg( \frac{1}{F'(\tilde\omega_\alpha)} - G'(\nu_c)\bigg) \big((\omega_\beta - \omega_\alpha) - (\tilde\omega_\beta - \tilde\omega_\alpha) \big) \bigg| \\
&\le \frac{p\kappa b_\Phi M_{ {\mathcal Q}[G^{\prime}}}{Nm_{{\mathcal Q}^{-1}[F]}} \mathcal{D}(\tilde\Omega^0) e^{-\Lambda_1t} \sum_{\alpha,\beta=1}^N |\omega_\alpha - \tilde\omega_\alpha|^{p-1}\big( |\omega_\alpha - \tilde\omega_\alpha| + |\omega_\beta - \tilde\omega_\beta| \big)\\
&\le \frac{6p\kappa^2 b_\Phi^2 M_{G'} M_{ {\mathcal Q}[G^{\prime}]}}{m_{{\mathcal Q}^{-1}[F]}} e^{-\Lambda_1t} \|\Omega - \tilde\Omega\|^{p}_p.
\end{aligned}
\end{align}
Here, the last inequality is followed from \eqref{C-6} and \eqref{Ap-5}. Finally, in \eqref{Ap-1},  we combine estimates in \eqref{Ap-5-1} and \eqref{Ap-6} to find 
\begin{align*}
\cI_{13} \leq -\frac{p\kappa G'(\nu_c) C_{131}}{M_{{\mathcal Q}^{-1}[F] }} \|\Omega - \tilde\Omega\|^p_p +\bigg( \frac{12p\kappa^2 b_\Phi^2 M_{G'}^2 M_{{\mathcal Q}[G^{\prime}]}}{m^2_{{\mathcal Q}^{-1}[F] }} +\frac{6p\kappa^2 b_\Phi^2 M_{G'} M_{{\mathcal Q}[G^{\prime}]}}{m_{{\mathcal Q}^{-1}[F] }} \bigg) e^{-\Lambda_1t} \|\Omega - \tilde\Omega\|^{p}_p,
\end{align*}
where $C_{131} = G'(\nu_c)m_\Phi\cos\theta_*$. This completes the proof.
\vspace{0.5cm}

\section{Proof of Proposition \ref{P4.1}} \label{App-B}
\setcounter{equation}{0}
We only verify the first part of Proposition \ref{P4.1}, since the second part can be treated using the same argument. We integrate the both sides of \eqref{sys:CGKM}$_1$ in time to obtain
\begin{equation*}
\theta(t,x) = \theta^0(x) + \int_0^t G \bigg( \nu(x) +\frac{\kappa}{\mu(D)}\int_D \phi(x,z)\sin(\theta(s,z)-\theta(s,x)) d\mu(z) \bigg) ds.
\end{equation*}
Before we define a solution space, we set two positive constants:
\begin{align*}
M :=  \|\nu\|_\infty +\frac{\kappa \|\phi\|_{L^{\infty}_x L^1_y}}{\mu(D)},\quad M^*:= \frac{\kappa }{\mu(D)} \Big( \sup_{-M<x<M} G'(x) \Big).
\end{align*}
Now, we define a closed subset $\mathcal{S}=\mathcal{S}_{\zeta,\theta^0}$ of the Banach space $\mathcal{C}\big( [0,\zeta]; L^\infty (D) \big)$ to find a local solution to the continuum GK model \eqref{sys:CGKM}:
$$ \mathcal{S}_{\zeta,\theta^0} := \big\{ \theta\in C([0,\zeta); L^\infty(D) \big) : \theta(0, x) = \theta^0(x), ~\mbox{a.e. } x\in D \big\}, \quad \zeta = \frac{1}{4M^* \|\phi\|_{L^{\infty}_x L^1_y}}, $$
equipped with the corresponding norm
$$\| \theta \|_{\mathcal{S}} := \sup_{0\le t \le \zeta} \| \theta(t,\cdot)\|_{\infty}.$$
It is easy to check that $\mathcal S$ is closed. Let $(\theta_n)_{n\geq1}$ be a convergent sequence in $\mathcal S$. Then, there exists $\theta^*\in\mathcal{C}\big( [0,\zeta); L^\infty (D) \big)$ such that
\begin{align*}
\lim_{n\to\infty} \sup_{0\le t\le \zeta} \|\theta_n(t,\cdot)-\theta^*(t,\cdot)\|_\infty = 0.
\end{align*}
This implies 
\[ \|\theta^0-\theta^*(0,\cdot)\|_\infty = 0 \]
so that $\theta^*\in\mathcal S$. \newline

Next, we introduce a nonlinear map $\mathcal{L}: {\mathcal S} \longrightarrow {\mathcal S}$ defined by
\begin{align} \label{Bp-1}
\begin{aligned}
\mathcal{L}[\theta](t,x) := \theta^0(x) + \int_0^t G \left( \nu(x) +\frac{\kappa}{\mu(D)}\int_D \phi(x,z)\sin(\theta(s,z)-\theta(s,x)) d\mu(z) \right) ds.
\end{aligned}
\end{align}
Note that the structure of $\mathcal L$ implies that a fixed point $\theta\in\mathcal S$ of $\mathcal L$ has $\mathcal C^1$-regularity with respect to time $t\in[0, \zeta]$ so that  $\theta$ becomes a local classical solution to the continuum GK model \eqref{sys:CGKM}. In the following lemma, we show that $\mathcal L$ is a contraction.
\begin{lemma} \label{LB.1}
$\mathcal{L}$ is a contraction on the closed subset $\mathcal{S}$. That is, there exists a constant $C\in(0, 1)$ such that
\begin{align*}
\begin{aligned}
\big\| \mathcal{L}[\theta_1]-\mathcal{L}[\theta_2] \big\|_{\mathcal{S}}
\le C \| \theta_1 - \theta_2 \|_{\mathcal{S}},
\end{aligned}
\end{align*}
for all $\theta_1, \theta_2 \in \mathcal S$.
\end{lemma}
\begin{proof}
For $x\in\Omega$ and $t\in[0,\zeta]$, we use the mean value theorem and \eqref{Bp-1} to obtain 
\begin{align*}
\begin{aligned}
& \big|\mathcal{L}[\theta_1](t,x)-\mathcal{L}[\theta_2](t,x)\big|  \\
&\hspace{0.5cm} \le M^* \int_0^t   \int_D \phi(x,z) \big|  \sin(\theta_1(s,z)-\theta_1(s,x)) -\sin(\theta_2(s,z)-\theta_2(s,x)) \big| d\mu(z) ds \\
&\hspace{.5cm} \le M^* \int_0^t   \int_D 2\phi(x,z) \bigg{|} \sin\left(\frac{\theta_1(s,z)-\theta_1(s,x)-\theta_2(s,z)+\theta_2(s,x)}{2}\right) \bigg{|} d\mu(z) ds \\
&\hspace{.5cm} \le M^* \int_0^t   \int_D \phi(x,z) \big{|} \theta_1(s,z)-\theta_2(s,z)+\theta_2(s,x)-\theta_1(s,x) \big{|} d\mu(z) ds \\
&\hspace{.5cm} \leq 2M^* \| \theta_1 - \theta_2 \|_{\mathcal S} \int_0^t \int_D \phi(x,z) d\mu(z) ds \\
&\hspace{.5cm} \leq  2\zeta M^* \|\phi\|_{L^{\infty}_x L^1_y} \| \theta_1 - \theta_2 \|_{\mathcal S} = \frac12 \| \theta_1 - \theta_2 \|_{\mathcal S}.
\end{aligned}
\end{align*}
This gives the desired result.
\end{proof}

\noindent {\it Proof of Proposition \ref{P4.1}}: We use Lemma \ref{LB.1} to obtain a unique fixed point of $\mathcal L$ in $\mathcal{S}$. This admits a local classical solution to the continuum GK model \eqref{sys:CGKM} over $D \times [0,\zeta]$. Since $\zeta$ is independent on the initial data, one can extend this local solution over next time-zone $D \times [\zeta ,2\zeta]$ using the same argument. In this way, we can extend local solution to any time interval to get the global solution. This completes the proof.

\section{Finite-in-time continuum limit} \label{App-C}
\setcounter{equation}{0}
In this appendix, we provide a proof of local-in-time continuum limit. 
\begin{theorem}\label{TC.1}
Suppose that $\kappa > 0$ and that there exists $b_\phi > 0$ such that $\|\phi\|_\infty < b_\phi$. Let $\theta \in \mathcal C^1 \big( [0, \infty); L^\infty(D) \big)$ be a global classical solution to the continuum GK model \eqref{sys:CGKM} given in Theorem \ref{T4.1} and let $\{\theta^N\}_{N\geq1}$ be a sequence of simple functions defined in \eqref{D-3}. Then for any finite $\tau>0$, we have
\begin{align*}
\begin{aligned}
\lim_{N\to\infty}\sup_{0 \leq t<\tau} \big\| \theta(t,\cdot)-\theta^N(t,\cdot)\big\|_{L^1} = 0.
\end{aligned}
\end{align*}
\end{theorem}
\begin{proof} Note that 
\begin{align}
\begin{aligned} \label{Ac-1}
&\partial_t \big( \theta(t, x) -\theta^N(t, x) \big) \\
& \hspace{0.5cm}  =
G\bigg(\nu(x) + \frac{\kappa}{\mu(D)}\int_D \phi(x,z) \sin\big(\theta(t, z) - \theta(t, x)\big)dz \bigg)\\
& \hspace{0.7cm} - G\bigg(\nu^N(x) + \frac{\kappa}{\mu(D)}\int_D \phi^N(x,z) \sin\big(\theta^N(t, z) - \theta^N(t, x)\big)dz \bigg) \\
& \hspace{0.5cm} \leq M_{G'} \bigg| \nu(x) -\nu^N(x) +\frac{\kappa}{\mu(D)} \int_D \big( \phi(x,z) -\phi^N(x,z)\big) \sin\big(\theta^N(t, z) - \theta^N(t, x)\big) dz \\
& \hspace{0.7cm}+\frac{\kappa}{\mu(D)} \int_D\phi(x,z)\big(\sin\big(\theta(t, z) - \theta(t, x)\big) - \sin\big(\theta^N(t, z) - \theta^N(t, x)\big)  \big) dz
\bigg|.
\end{aligned}
\end{align}
We multiply the both side by $\sgn\big(\theta(t, x)-\theta^N(t, x)\big)$ and integrate \eqref{Ac-1} over $D$ to get
\begin{align*}
\begin{aligned}
\partial_t \big\| \theta(t, \cdot) -\theta^N(t, \cdot) \big\|_{L^1} \le M_{G'} \bigg( \|\nu - \nu^N\|_{L^1}+\frac{\kappa}{\mu(D)}\|\phi - \phi^N\|_{L^1} +2b_\phi \kappa\| \theta(t, \cdot) - \theta^N(t, \cdot)\|_{L^1} \bigg).
\end{aligned}
\end{align*}
This implies
\[  \sup_{0<t<\tau} \|\theta - \theta^N\|_{L^1} \le e^{a \tau} \| \theta^{N,0} - \theta^0 \|_{L^1} + \frac{b}{a}(e^{a \tau}-1), \]
where 
\[  a=2M_{G'} b_\phi \kappa, \quad b=M_{G'}\bigg( \|\nu-\nu^N\|_{L^1} + \frac{\kappa}{\mu(D)}\| \phi - \phi^N\|_{L^1} \bigg). \]
We use Lebesgue differentiation theorem to get
\begin{align*}
\lim_{N\to\infty} \| \theta^{N,0} - \theta^0 \|_{L^1} = 0, \quad \lim_{N\to\infty}b=0,
\end{align*}
which implies our desired result.
\end{proof}

\bibliographystyle{amsplain}

\end{document}